\documentclass{amsart}
\usepackage{amssymb}
\usepackage[all]{xy}
\usepackage[utf8x]{inputenc}

\newtheorem{Proposition}{Proposition}[section]
\newtheorem{Lemma}[Proposition]{Lemma}
\newtheorem{Definition}[Proposition]{Definition}
\newtheorem{Corollary}[Proposition]{Corollary}
\newtheorem{Theorem}[Proposition]{Theorem}
\newtheorem{Claim}[Proposition]{Claim}

\theoremstyle{definition}
\newtheorem{Remark}[Proposition]{Remark}
\newtheorem{Example}[Proposition]{Example}

\renewcommand{\AA}{\mathbb{A}}
\newcommand{\BB}{k}
\newcommand{\Base}{\mathfrak{B}}

\newcommand{\GG}{\mathbb{G}}
\newcommand{\LL}{\mathbb{L}}
\newcommand{\MM}{\mathfrak{M}}

\newcommand{\PP}{\mathbb{P}}
\newcommand{\QQ}{\mathbb{Q}}
\renewcommand{\SS}{\mathfrak{M}}
\newcommand{\TT}{\mathfrak{N}}
\newcommand{\WW}{\mathcal{W}}
\newcommand{\XX}{\mathcal{X}}
\newcommand{\YY}{\mathcal{Y}}
\newcommand{\ZZ}{\mathcal{Z}}
\newcommand{\muu}{\hat{\mu}}
\newcommand{\Nn}{N}

\DeclareMathOperator{\KK}{K}
\DeclareMathOperator{\Var}{Var}
\DeclareMathOperator{\Sch}{Sch}
\DeclareMathOperator{\St}{St}
\DeclareMathOperator{\Sta}{St^{aff}}
\DeclareMathOperator{\Sp}{Sp}
\DeclareMathOperator{\Gl}{Gl}
\DeclareMathOperator{\gl}{gl}
\DeclareMathOperator{\BGl}{BGl}
\DeclareMathOperator{\Spec}{Spec}
\DeclareMathOperator{\Sym}{Sym}
\DeclareMathOperator{\SSym}{\Sym}
\DeclareMathOperator{\Symm}{\mathcal{S}ym}
\DeclareMathOperator{\SSymm}{\Symm}
\DeclareMathOperator{\cl}{cl}
\DeclareMathOperator{\Char}{char}
\DeclareMathOperator{\Ho}{H}

\DeclareMathOperator{\DER}{D}
\DeclareMathOperator{\Rep}{Rep}
\DeclareMathOperator{\Vect}{Vect}
\DeclareMathOperator{\Lo}{\mathcal{L}}
\DeclareMathOperator{\Crit}{Crit}
\DeclareMathOperator{\Hom}{Hom}
\DeclareMathOperator{\Mat}{Mat}
\DeclareMathOperator{\Span}{span}
\DeclareMathOperator{\Hilb}{Hilb}
\DeclareMathOperator{\tr}{tr}

\DeclareMathOperator{\lcm}{lcm}

\DeclareMathOperator{\CF}{CF}

\title{Motivic DT-invariants for the one loop quiver with potential}
\author{Ben Davison, Sven Meinhardt}

\begin{document}

\maketitle

\begin{abstract}
In this paper we compute the motivic Donaldson--Thomas invariants for the quiver with one loop and any potential. As the presence of arbitrary potentials requires the full machinery of $\muu$-equivariant motives, we give a detailed account of them. In particular, we will prove two results for the motivic vanishing cycle which might be of importance not only in Donaldson--Thomas theory. 
\end{abstract}

\tableofcontents

\section{Introduction}

Donaldson--Thomas invariants were first introduced by R.\ Thomas \cite{Thomas1} to give an alternative way to count (irreducible) curves on a Calabi--Yau 3-fold. A few years later D.\ Joyce \cite{JoyceI},\cite{JoyceII},\cite{JoyceIII},\cite{JoyceMF},\cite{JoyceIV} and Y. Song \cite{JoyceDT} generalized the definition to much more general situations using results of K.\ Behrend \cite{Behrend}. Shortly after this M.\ Kontsevich and Y.\ Soibelman \cite{KS1},\cite{KS3} came up with an alternative definition which turns out to be equivalent to the one given by Joyce (see \cite{Meinhardt3}). It has been realized subsequently by several people that the Donaldson--Thomas invariants should be of motivic origin, in other words, they should be Euler characteristics of certain motives. Among several papers giving a definition of motivic Donaldson--Thomas invariants (at least in special cases), we will basically follow \cite{Meinhardt3}, but the reader is also encouraged to consult \cite{BBS},\cite{KS2},\cite{Mozgovoy1}. There are only a few nontrivial examples, where motivic Donaldson--Thomas invariants have been computed for all classes in the Grothendieck group of the Calabi--Yau 3-category in question, though let us mention the papers \cite{BBS} and \cite{MMNS}. We will start from the rather trivial example of coherent sheaves of dimension zero on $\AA^1_k$ with $k=\bar{k}$ and $\Char k=0$ or, equivalently, finite dimensional representations of $k[t]$, i.e.\ finite dimensional representations of the one loop quiver. The corresponding category is of homological dimension one but is also the heart of a bounded t-structure in a Calabi--Yau 3-category. The motivic Donaldson--Thomas invariants for this example are well-known and given by 
\[ \Omega_n=\begin{cases}
             \LL^{1/2} & \mbox{ for } n=1 \mbox{ and} \\ 0 & \mbox{ else},
            \end{cases} \]
where $\LL^{1/2}=1-[\mu_2]$ denotes a square root of the motive $\LL$ of $\AA^1_k$ in $\KK^{\muu}(\Var/k)$ having Euler characteristic $-1$. Moreover, $[\mu_d]$ denotes the $\muu$-equivariant motive of the set of $d$-th roots of unity with obvious $\muu$-action. 
We will modify this example by considering a nonzero ``potential'' $W\in k[t]$ and by requiring that the sheaves are annihilated by multiplication with $W'=dW/dt$, i.e.\ are supported on $\Spec (k[t]/(W'))$. In the language of representations the operator $A$ given by $t$ has to satisfy $W'(A)=0$. This class of examples is not of great importance but is nevertheless interesting as it provides the first case in which the full machinery of $\muu$-equivariant motives has to be applied. Here is our main result.  
\begin{Theorem} \label{maintheorem} For $W\in k[T]$ let $W'=c\prod_{i=1}^r(t-a_i)^{d_i-1}$ be the prime decomposition of $W'$ into linear factors with $c\in k^\times$, $1< d_i\in \mathbb{N}$ and $a_i\in k$ for all $1\le i\le r$. Define the Donaldson--Thomas invariants $\Omega_{\vec n}\in \KK^{\muu}(\Sta/k)$ for any $r$-tuple $(n_1,\ldots,n_r)\in \mathbb{N}^r$ as in section 6.  Then
\[ \Omega_{\vec n}= \begin{cases} 
              \LL^{-1/2}\bigl(1- [\mu_{d_i}]\bigr) &\mbox{ for } \vec n=\vec e_i\mbox{ being the i-th basis vector } (1\le i\le r),  \\ 0 & \mbox{ otherwise}.
             \end{cases}
\]
In particular, $\Omega_{\vec n}$ is in the image of $\KK^{\muu}(\Var/k)[\LL^{-1/2}]$ in $\KK^{\muu}(\Sta/k)$.
\end{Theorem}
By taking the Euler characteristic of the nontrivial $\Omega_{\vec n}$, we end up with the Milnor numbers $d_i-1$ of $W$ at the critical points $a_i\in \AA^1_k$. \\
Using a standard trick for computing Donaldson--Thomas invariants (cf.\ Lemma \ref{wallcrossing}) and $\Hilb^n \AA^1_k=\Sym^n \AA^1_k$, the main theorem would actually be an easy application of the following claim concerning the commutativity of the motivic vanishing cycle and the functor $\Sym^n$.  
\begin{Claim}  
Let $f:X\longrightarrow \AA^1_k$ be a regular map on a smooth variety of dimension $d$. Denote by $X_0$ the fiber over zero. For any $n\ge 0$ there is an obvious map $\SSym_+^n(f):\Sym^n X \xrightarrow{\Sym^n(f)} \Sym^n\AA^1_k\xrightarrow{+} \AA^1$. If we consider the motivic vanishing cycles $\phi_f\in \KK^{\muu}(\Var/X_0)[\LL^{-1/2}]$ and $\phi_{\SSym_+^n(f)}|_{\Sym^n X_0}\in \KK^{\muu}(\Var/\Sym^n X_0)[\LL^{-1/2}]$ as elements of the $\lambda$-ring $\KK^{\muu}(\Var/\Sym X_0)[\LL^{-1/2}]$ via the obvious inclusions, we get for any $n\ge 0$ the equation
\[ \LL^{nd/2} \phi_{\SSym_+^n(f)}|_{\Sym^n X_0} = \sigma^n ( \LL^{d/2} \phi_f) \]
with $\sigma^n(-)$ $(n\in \mathbb{N})$ defining the $\lambda$-ring structure.
\end{Claim}
The attentive reader might have realized that $\phi_{\SSym_+^n(f)}$ is only defined for smooth varieties $\Sym^n X$, while $\Sym^n X$ is in general not smooth. Unfortunately, we were not able to prove this result even for smooth symmetric products. There is, however, a similar form with $\Sym^n X$ replaced by the smooth Deligne--Mumford stack $\Symm^n X:=X^n/S_n$. We will prove this ``stacky'' version in section 5 which can also be seen as a generalization of the famous Thom--Sebastaini Theorem. \\
Our strategy to prove this result is by relating the vanishing cycle to the (the integral over) a much simpler vanishing cycle functor defined  only for $\GG_m$-equivariant functions. We will prove a lot of nice properties for this ``equivariant'' vanishing cycle functor including the equivariant version of the above theorem (regardless of smoothness of $\Sym^n X$) and the philosophy is that all properties of the usual vanishing cycle should actually follow from their counterparts in the equivariant setting. This will be shown in section 5. The close relationship culminates in the following theorem saying that both versions of the vanishing cycle coincide (after integration) for $\GG_m$-equivariant functions. The proof has been sketched to the second author by Dominic Joyce in a private communication. All credits are, therefore, attributed to him and all errors to us.
\begin{Theorem} \label{theorem0} 
Let $X$ be a smooth variety with $\GG_m$-action such that every point has an open neighborhood isomorphic to $\AA^r_k\times Z$ with $\GG_m$ acting via $g\cdot(v_1,\ldots,v_r,z)=(gv_1,\ldots, gv_r,z)$ for all $g\in \GG_m, (v_1,\ldots,v_r)\in \AA^r_k$ and $z\in Z$. Let $f:X \rightarrow \AA_k^1$ be a $\GG_m$-equivariant morphism of degree $d>0$, i.e.\ $f(g\cdot x)=g^df(x) \;\forall g\in\GG_m, x\in X$ and let $\muu$ act on $f^{-1}(1)$ via $\mu_d$ and trivially on $f^{-1}(0)$. Then
 \[ \int_{X} \phi_f = \LL^{-\frac{\dim X}{2}}\bigl([f^{-1}(0)] - [f^{-1}(1)] \bigr) \;\;\mbox{in} \;\KK_0^{\muu}(\Var/k)[\LL^{-1/2}].\]
\end{Theorem}
This result should be compared to a similar result in \cite{BBS}, where the assumptions on the action are less strict but the assumptions on $f$ are much more restrictive. We actually believe that the theorem remains true if we allow a $\GG_m$-action on $\AA^r_k$ with any positive weights $0<w_1,\ldots,w_r$. \\
Having this theorem at hand, one only has to spell out the definitions and to prove a kind of perturbation lemma allowing us to reduce ourselves to homogeneous potentials in order to give a proof of Theorem \ref{maintheorem}. \\

The strategy of the paper is as follows. In section 2 and 3 we recall the basic definitions and properties of motives. In particular, we establish the $\lambda$-ring structure on the ring $\KK_0(\Var/k)$ of motives. In section 4 we construct an exotic $\lambda$-ring structure on $\KK_0^{\muu}(\Var/k)$ following ideas of Kontsevich and Soibelman \cite{KS1}. Moreover, we introduce a naive version of (the integral over) the vanishing cycle sheaf for $\GG_m$-equivariant functions as mentioned above, which is basically the expression on the right hand side of Theorem \ref{theorem0}. In section 5 we present a new approach to the vanishing cycle sheaf closely related to that of Denef and Loeser \cite{DenefLoeser2},\cite{DenefLoeser1} which allows us to reduce the main properties to their counterparts in the equivariant case proven in section 4. Section 5 closes with the proof of Theorem \ref{theorem0}. The last section concerns the one loop quiver and its theory of (motivic) Donaldson--Thomas invariants. By proving a few simple lemmas we can finally show that our main theorem holds.\\

\textbf{Acknowledgments.} The first author would like to thank Bal\'{a}zs Szendr\H{o}i for teaching him about motivic vanishing cycles, and for his support, in general, over the last few years.  The first author's work was conducted partly during stays at the University of British Columbia, where he would like to thank everyone, especially Jim Bryan, for providing wonderful research conditions, and secondly at the University of Strathmore, Nairobi, which he would also like to thank for providing an excellent environment for Mathematics.  Finally he would like to thank Ezra Getzler, for stimulating conversations about lambda rings and motives.  As mentioned earlier, the second author is more than grateful to Dominic Joyce, not only for sharing his ideas leading to the proof of Theorem \ref{theorem0}, but also for his ceaseless patience in teaching and explaining the deep theory of motivic Donaldson--Thomas invariants to him. Without this, the following paper would not exist at all. The second author would also like to thank Maxim Kontsevich for inviting him to the IHES and for some helpful discussions clarifying the ``big picture''. The second author's research was supported by the EPSRC grants EP/D07790/1 and EP/G068798/1, by the IHES, notably by the Klaus Tschira foundation, and finally by the University of Bonn and the SFB/TR 45.\\

\section{Motives}

We will start by recalling a few standard facts about the group of (naive) motives. Let $\SS$ be an Artin stack locally of finite type over some field $k=\bar{k}$ of characteristic zero.

\begin{Definition}
Define $\KK_0(\Var/\SS)$ to be the abelian group generated by isomorphism classes of morphisms $X\xrightarrow{f} \SS$ of finite type with $X$ being a reduced separated scheme over $k$ subject to the relation
\[ [X\xrightarrow{f} \SS ] \sim [Z\xrightarrow{f|_Z} \SS ] + [X\setminus Z \xrightarrow{f|_{X\setminus Z}} \SS] \]
for any closed and reduced subscheme $Z\subset X$.
\end{Definition}

For any $k$-morphism $\SS \xrightarrow{\pi} \TT$ of finite type there is an induced group homomorphism
\[ \pi_\ast: \KK_0(\Var/\SS) \longrightarrow \KK_0(\Var/\TT) \]
also denoted by $\int_\pi$ which is injective for any locally closed substack $\SS \hookrightarrow \TT$. Indeed, $\pi^\ast$ defined by linear extension of $\pi^\ast([X \xrightarrow{f} \TT]) = [f^{-1}(\SS) \rightarrow \SS]$ satisfies $\pi^\ast \circ \pi_\ast = id$. \\

\begin{Remark} \label{Kaiso} \qquad
\begin{enumerate} 
\item One can replace the category $\Var/\SS$ by the category $\Sch/\SS$ of finite type $\SS$-schemes or $\Sp/\SS$ of finite type algebraic $\SS$-spaces. The induced group homomorphisms
\[ \KK_0(\Var/\SS) \longrightarrow \KK_0(\Sch/\SS) \longrightarrow \KK_0(\Sp/\SS) \]
are isomorphisms if the connected components of $\SS$ are of finite type. In all other cases one should pass to the completion with respect to the topology having $\KK_0(\Var/\SS\setminus\mathcal{U}) \subset \KK_0(\Var/\SS)$ as neighborhoods of zero, where $\mathcal{U}\subset \SS$ runs through the directed set of open substacks of finite type over $k$. The completion is already given by the quotient 
\[ \hat{\KK}_0(\Var/\SS) = \KK_0(\Var/\SS) / \cap_{\mathcal{U}\subset \SS} \KK_0(\Var/\SS \setminus \mathcal{U}) \]
making $\KK_0(\Var/\SS)$ into a Hausdorff space. Notice that the intersection $\cap_{\mathcal{U}\subset \SS} \KK_0(\Var/\SS \setminus \mathcal{U})$ is nonzero even though $\cup_{\mathcal{U}\subset \SS} \mathcal{U}=\SS$. Indeed, if $(\ZZ_n)_{n\in \mathbb{N}}$ is a locally finite stratification of a connected stack $\SS$ into locally closed substacks $\ZZ_n$ of finite type, the element $0\neq [\SS \rightarrow \SS]-[\sqcup_{n\in \mathbb{N}} \ZZ_n \rightarrow \SS]$ is in the intersection. A similar construction for $\Sch/\SS$ and $\Sp/\SS$ leads to
\[ \hat{\KK}_0(\Var/\SS) \xrightarrow{\sim} \hat{\KK}_0(\Sch/\SS) \xrightarrow{\sim} \hat{\KK}_0(\Sp/\SS). \]

\item By taking fiber products there is a well-defined pull-back homomorphism
\[ \pi^\ast: \KK_0(\Sp/\TT) \longrightarrow \KK_0(\Sp/\SS) \]
for any representable morphism $\pi:\SS \rightarrow \TT$. After passing to completions or under special conditions on $\SS$ (see above), we can replace $\KK_0(\Sp/\dots)$ by $\KK_0(\Var/\ldots)$.
\end{enumerate}
\end{Remark}

Let $\SS$ be a monoid in the category of Artin stacks locally of finite type over $k$, i.e.\ there are $k$-morphisms $\epsilon:\Spec(k)\rightarrow \SS$ and $\mu:\SS\times_k \SS \rightarrow \SS$ satisfying the usual axioms of a monoid. Let us assume that $\mu$ is of finite type. The following lemma is obvious.

\begin{Lemma}
 There is a convolution product on $\KK_0(\Var/\SS)$ defined by bilinear extension of
\[ [X\xrightarrow{f} \SS]\cdot [Y \xrightarrow{g} \SS] = [ X\times_k Y \xrightarrow{f\times g} \SS\times_k \SS \xrightarrow{\mu} \SS]. \]
Together with the unit $1=[\Spec(k)\xrightarrow{\epsilon} \SS]$ it provides $\KK_0(\Var/\SS)$ with a ring structure which is commutative if $\mu$ is.
\end{Lemma}

\begin{Remark} \qquad
\begin{enumerate}
\item There is also a fiber product (over $\SS$) on $\KK_0(\Sp/\SS)$ for any Artin stack $\SS$ locally of finite type. It can be seen as a special case of a more general convolution product (see the end of the next section).  
\item If $\pi:\SS \rightarrow \TT$ is a finite type homomorphism of monoids in the category of Artin stacks locally of finite type over $k$, $\pi_\ast$ is a homomorphism of rings.
\end{enumerate}
\end{Remark}

\begin{Example} \label{Gln} Let us denote the motive of $\AA_k^1$ by $\LL$ and that of $\PP_k^{n-1}$ by $[n]$. Hence, $[n]=1+\LL + \ldots + \LL^{n-1}=\frac{\LL^{n}-1 }{ \LL-1}$ in the ring\footnote{Notice that $\Spec(k)$ is a monoid.} $\KK_0(\Var/k)$. If we denote the product $\prod_{i=1}^n [i]$ by $[n]!$, we also get $[\Gl_k(n)]=\prod_{i=0}^{n-1}(\LL^n-\LL^i)=\LL^{{n \choose 2}}[n]!(\LL-1)^n$ (see Lemma \ref{principal}).
\end{Example}

All of the above can be generalized to objects $\XX \rightarrow \SS$ in the category $\St/\SS$ of Artin stacks of finite type over $\SS$. This leads directly to the group $\KK_0(\St/\SS)$ which is a (commutative) ring if $\SS$ is a (commutative) monoid. Moreover, we can consider the full subcategory $\Sta/\SS$ of $\St/\SS$ consisting of morphisms $\XX \rightarrow \SS$ as above with $\XX$ having affine (or equivalently ``linear'') groups as stabilizers of closed points. \\
If $\SS$ is a monoid, we get the following sequence of ring homomorphisms
\begin{eqnarray} \label{kette}
 &&\KK_0(\Var/k) \xrightarrow{\epsilon_\ast}  \KK_0(\Var/\SS) \longrightarrow \KK_0(\Sch/\SS) \longrightarrow  \\
 && \qquad \qquad \qquad \longrightarrow \KK_0(\Sp/\SS) \longrightarrow \KK_0(\Sta/\SS) \longrightarrow \KK_0(\St/\SS). \nonumber 
\end{eqnarray}

There is also an equivariant version of the theory. For this let $G$ be an affine group acting on all varieties and stacks in question. Moreover, we assume that all morphisms are $G$-equivariant and that the action is good in the sense that every point has an affine $G$-invariant neighborhood. Denote by $\KK_0^G(\St/\SS)$ the corresponding group of motives, hence $\KK_0^{\{1\}}(\St/\SS)=\KK_0(\St/\SS)$. We are mostly interested in the group\footnote{One should probably write $\GG_{m,k}$ to distinguish this object from the functor $\GG_m:\SS\mapsto \Gl_\SS(1)$ but by abuse of language we will skip the index $k$.} $\GG_m$ and its subgroups
$\mu_d$ of $d$-th roots of unity.
For technical reasons it is useful to consider the quotient of $\KK_0^G(\St/\SS)$ with respect to the subgroup generated by
\begin{equation} \label{relation} [ \XX \xrightarrow{\pi} \mathcal{Y} \xrightarrow{f} \SS] -\LL^r\cdot[\mathcal{Y} \xrightarrow{f}\SS] \end{equation}
for any $G$-equivariant vector bundle $\XX\xrightarrow{\pi}\mathcal{Y}$ on $\mathcal{Y}$ of rank $r$. The quotient will be denoted by $\KK^G(\St/\SS)$ and by $\KK(\St/\SS)$ if $G$ is trivial\footnote{In \cite{Ekedahl2} and \cite{Ekedahl1} this group is denoted by $\KK_0(\rm Stck_k)$ for $\SS=\Spec(k)$.}. A similar construction works for any subcategory in (\ref{kette}) leading to quotient groups $\KK^G(\ldots/\SS)$ which are in fact rings as the subgroup generated by (\ref{relation}) is actually an ideal. Notice that for $G=\{1\}$ and $\SS$ a variety the element in (\ref{relation}) is automatically zero in the first four $\KK_0$-groups in (\ref{kette}) and, thus, for example $\KK(\Var/\SS)=\KK_0(\Var/\SS)$.\\ 

The following lemma will be useful throughout the paper. Notice that the ring $\KK_0(\Var/k)$ maps naturally into $\KK^{G}(\St/\SS)$ as any variety carries a trivial $G$-action. 

\begin{Lemma}[cf.\ \cite{Ekedahl1}, Proposition 1.1] \label{principal}
Let $\XX\xrightarrow{\pi} \YY$ be a $G$-equivariant $\Gl_k(n)$-principal bundle on $\YY \xrightarrow{f}\SS$. Then $[\XX \xrightarrow{\pi} \YY \xrightarrow{f} \SS]=[\Gl_k(n)]\cdot[\YY\xrightarrow{f} \SS]$ in $\KK^{G}(\St/\SS)$.
\end{Lemma}
\begin{proof}
Let $\mathcal{V} \rightarrow \YY$ be the $G$-equivariant vector bundle on $\YY$ associated to $\pi$, i.e.\ $\mathcal{V}=\XX\times_{\Gl_k(n)} \AA^n_k$ with $G$ acting on the first factor only. For $0\le i\le n$ let $\XX_i\rightarrow \YY$ be the bundle of $i$ linear independent vectors in $\mathcal{V}$ with $\XX_0:=\YY$. Obviously $\XX_n=\XX$ and there are natural $G$-equivariant projections $\XX=\XX_n \rightarrow \XX_{n-1} \rightarrow \ldots \rightarrow \XX_0=\YY$. For any $0\le i < n$ the bundle $\XX_{i+1} \rightarrow\XX_i$ is the complement in $\mathcal{V}\times_\YY \XX_i$ of the canonical subbundle of rank $i$ spanned by the vectors in $\XX_i$. This bundle is actually a $G$-equivariant vector bundle of rank $i$. By (\ref{relation}), Example \ref{Gln} and the scissor relations $[\XX\rightarrow \SS]=[\YY\rightarrow \SS]\cdot\prod_{i=0}^{n-1}(\LL^n-\LL^i)=[\Gl_k(n)]\cdot[\YY\rightarrow \SS]$. \end{proof}
 
\begin{Corollary} 
The image of $[\Gl_k(n)]$ in $\KK^G(\Sta/\SS)$ is invertible with inverse $[\BGl_k(n) \rightarrow \Spec(k) \xrightarrow{\epsilon} \SS]$, where $\BGl_k(n)$ denotes the quotient stack $\Spec(k)/\Gl_k(n)$.  
\end{Corollary}

The following proposition generalizes Remark \ref{Kaiso} (1).
\begin{Proposition}[cf.\ \cite{Ekedahl1}, Theorem 1.2] \label{Kaiso2}
If $\SS$ is a monoidal Artin stack of finite type over $k$, the ring homomorphism
\[ \KK^G(\Var/\SS)\big[[\Gl_k(n)]^{-1}, n\in \mathbb{N}\big] \longrightarrow \KK^G(\Sta/\SS) \]
is an isomorphism. For $\SS$ being locally of finite type over $k$ one needs to complete $\KK^G(\Var/\SS)\big[[\Gl_k(n)]^{-1}, n\in \mathbb{N}\big]$ with respect to the topology introduced in Remark \ref{Kaiso} (1).
\end{Proposition}

\begin{Remark} \label{Kaiso4}
Let us assume that $G:=\GG_m$ acts trivially on $\SS$. Notice that $\KK(\St/\SS) \hookrightarrow \KK^{\GG_m}(\St/\SS)$ by taking trivial $\GG_m$-actions. A left inverse is given by forgetting the $G$-action. Any $\GG_m$-equivariant stack $\XX \rightarrow \SS$ is a $\GG_m$-principal bundle on $\XX/\GG_m\longrightarrow \SS$ with trivial $\GG_m$-action. Hence $\KK^{\GG_m}(\St/\SS)=\KK(\St/\SS)$ by Lemma \ref{principal}. Now, let $X\rightarrow \SS$ be a $\GG_m$-equivariant variety over $\SS$ together with an affine $\GG_m$-equivariant stratification $(X_i)_{i\in I}$ of $X$ which exists by the goodness of the $\GG_m$-action. Moreover, we can assume that the stabilizer is constant on every stratum $X_i$. As the field $k$ is algebraically closed, we can apply Luna's \'{e}tale slice theorem to any stratum to see that $X_i$ is a $\GG_m$-equivariant principal bundle on $X_i//\GG_m$ with structure group being the trivial group $\{1\}$ or $\GG_m$ depending on the size of the stabilizer. Thus, by Lemma \ref{principal} $[X\rightarrow \SS]$ is in the image of $\KK(\Var/\SS)$ in $\KK^{\GG_m}(\Var/\SS)$ and $\KK(\Var/\SS)=\KK^{\GG_m}(\Var/\SS)$ follows as in the ``stacky'' case.    
\end{Remark}

\section{$\lambda$-rings}

The notion of a $\lambda$-ring is crucial for defining (motivic) Donaldson--Thomas invariants. As the theory is well-known (cf.\ \cite{Knutson},\cite{Yau}), we will only recall the definitions and give a few examples.
\begin{Definition}
A commutative ring $R$ is called a $\lambda$-ring if it is equipped with an additional map $\lambda:R\ni a \longmapsto \lambda_a(t)=\sum_{n\ge 0} \lambda^n(a)t^n \in R[[t]]$ such that
\begin{itemize}
 \item[(i)] $ \lambda_a(t)= 1 + at \mod t^2$,
 \item[(ii)] $ \lambda_0(t)=1$,
 \item[(iii)] $ \lambda_{a+b}(t)=\lambda_a(t)\lambda_b(t).$
\end{itemize}
The opposite $\lambda$-ring $(R,\lambda^{op})$ is defined by $\lambda^{op}_a(t)=\lambda_a(-t)^{-1}$. A homomorphism from a $\lambda$-ring $(R_1,\lambda^{(1)})$ to a $\lambda$-ring $(R_2,\lambda^{(2)})$ is a ring homomorphism $\pi:R_1 \rightarrow R_2$ such that $ \pi(\lambda^{(1)}_a(t)) = \lambda^{(2)}_{\pi(a)}(t)$ for all $a\in R_1$. 
\end{Definition}

\begin{Remark}
The letter $\lambda$ has been chosen as in many early examples the structure was induced by exterior powers $\Lambda$ (see below). The opposite $\lambda$-structure was than denoted by the letter $\sigma$ as it was induced by symmetric powers $\Sym$. In the sequel we will use the letter $\sigma$ to denote a general $\lambda$-structure and reserve the letter $\lambda$ for its opposite $\sigma^{op}$ because in our examples the $\lambda$-structure is mostly induced by operations like $\Sym$.    
\end{Remark}
 
\begin{Definition} \qquad \label{special}
\begin{itemize}
\item[(i)] A $\lambda$-ring $(R,\sigma)$ is called special if $\sigma^n(ab)=P^n(\sigma^1(a),\ldots, \sigma^n(a),\sigma^1(b),\ldots,$ $ \sigma^n(b))$ and $\sigma^m(\sigma^n(a))=Q^{m,n}(\sigma^1(a),\ldots,\sigma^{mn}(a))$ for certain universal polynomials $P^n\in \mathbb{Z}[X_1,\ldots,X_n,Y_1,\ldots,Y_n], Q^{m,n}\in \mathbb{Z}[X_1,\ldots,X_{mn}]$ independent of $(R,\sigma)$ and $a\in R$.
\item[(ii)] An element $a\in R$ is called a line element if $\sigma_a(t)=\sum_{n\ge 0}a^nt^n=\frac{1}{1-at}$. 
\end{itemize}
\end{Definition}

\begin{Remark} \qquad
\begin{enumerate}
\item The Definition \ref{special} can be also expressed in terms of $\lambda=\sigma^{op}$ (with different universal polynomials), e.g.\ $a\in R$ is a line element if $\lambda_a(t)=1+at$. These ``opposite'' definitions are often used in the literature, in other words $(R,\sigma)$ is special in our sense if $(R,\sigma^{op})$ is special in terms of the opposite definition and similarly for line elements.
\item Some authors call a $\lambda$-ring a pre-$\lambda$-ring and a special $\lambda$-ring just a $\lambda$-ring. One should, therefore, be careful with the literature.
\end{enumerate}
\end{Remark}


\begin{Example} \label{lambdarings}\qquad
\begin{enumerate}
\item The standard $\lambda$-structures on $\mathbb{Z}$ are $\sigma_a(t)=\frac{1}{(1-t)^a}=\sum_{n\ge 0} { a + n- 1 \choose n} t^n$ and $\lambda_a(t)=\sigma^{op}_a(t)=(1+t)^a=\sum_{n\ge 0} { a \choose n}t^n$. 
The only line element is $1\in \mathbb{Z}$. 
 
\item Let $G$ be some algebraic group over $\QQ$ and denote by $R$ the $\KK_0$-group of the abelian category $G-\Rep_\QQ^{f.d.}$ of finite dimensional $\QQ$-linear algebraic $G$-representations. There is a $\lambda$-ring structure $\sigma$ on $R$ such that if $V$ is a finite dimensional $G$-representation, $\sigma^n(\cl(V))=\cl(\Sym^nV)$ and $\lambda^n(\cl(V))=\cl(\Lambda^n V)$ for the opposite $\lambda$-structure $\lambda=\sigma^{op}$. The $\lambda$-ring $(R,\sigma)$ is special and classes of one-dimensional representations are line elements. For $G=\{1\}$ we obtain the previous Example.
  
\item Let $R=\CF(\SS,\mathbb{Z})$ be the space of constructible $\mathbb{Z}$-valued functions on a commutative monoidal algebraic space $\SS$ with multiplication map $\mu:\SS\times \SS \rightarrow \SS$ being of finite type. The ring structure on $R$ is given by the convolution product $fg=\mu_\ast(f\boxtimes g)$, where the push-forward is given by integrating along the fibers with respect to the Euler characteristic with compact support. Any constructible function $f$ on $\SS$ defines a constructible function on $\Sym^n \SS$ with value $\prod_{i=1}^p {f(x_{m_i}) + m_{i+1}-m_i-1 \choose m_{i+1}-m_i}$ at $(x_1,\ldots,x_n)\in \Sym^n \SS$ satisfying $x_m=x_{m_i}$ if and only if $m_i\le m < m_{i+1}$. By integrating this function along the map $\Sym^n \SS \xrightarrow{\:\mu\:}\SS$ we obtain the constructible function $\sigma^n(f)$. This defines a special $\lambda$-ring structure $\sigma$ on $R$. Moreover, any homomorphism $\SS \rightarrow \TT$ of finite type defines a $\lambda$-ring homomorphism by pushing forward functions. For $\SS=\Spec(k)$ we obtain once more the first Example.
 
\item Let $(\SS,\mu,\epsilon)$ be commutative monoid in the category of algebraic spaces locally of finite type over $k$ with multiplication map $\mu:\SS\times \SS \rightarrow \SS$ being of finite type as before. There is a $\lambda$-ring structure $\sigma$ on $R=\KK_0(\Var/\SS)$ such that
\[ \sigma^n([X\xrightarrow{f} \SS])=\big[\Sym^n X \xrightarrow{\Sym^n(f)} \Sym^n \SS \xrightarrow{\mu} \SS \big] \]
This $\lambda$-ring is not special (see \cite{LarsenLunts}, section 8). However, $\LL=[\AA^1\rightarrow \Spec(k)\xrightarrow{\epsilon} \SS]$ is a line element and for all $m,n\ge 1$ and any polynomial $f\in \mathbb{Z}[X]$ the following equations hold
\begin{eqnarray*}
 \sigma^n(f(\LL) a)&=& P^n(\sigma^1(a),\ldots, \sigma^n(a),\sigma^1(f(\LL)),\ldots, \sigma^n(f(\LL))), \\
 \sigma^m(\sigma^n(f(\LL))) &=& Q^{m,n}(\sigma^1(f(\LL)),\ldots,\sigma^{mn}(f(\LL))).
\end{eqnarray*}
Using this one can extend the $\lambda$-structure to the localization of $\KK_0(\Var/\SS)$ with respect to any family $\mathcal{F}$ of polynomials $f\in \mathbb{Z}[X]$ such that $f(X)\in \mathcal{F}$ implies $f(X^n)\in \mathcal{F}$ for all $n\ge 1$ by putting 
\[ \hspace{0.8cm} \sigma^n\Bigl(\frac{a}{f(\LL)}\Bigr):=P^n(\sigma^1(a),\ldots, \sigma^n(a),\sigma^1(f(\LL)^{-1}),\ldots, \sigma^n(f(\LL)^{-1})),\]
where the $\sigma^n(f(\LL)^{-1})$ are uniquely determined by 
\[ \hspace{1.2cm} P^n(\sigma^1(f(\LL)^{-1}),\ldots, \sigma^n(f(\LL)^{-1}),\sigma^1(f(\LL)),\ldots, \sigma^n(f(\LL)))=\sigma^n(1)=1 \] 
as the coefficient of $\sigma^n(f(\LL)^{-1})$ can be shown to be $f(\LL^n)$. If we apply this to the family $\mathcal{F}=\{\LL^n,n\in \mathbb{N}\}\cup \{\LL^n-1,n\in \mathbb{N}\}$ we get by Example \ref{Gln} a $\lambda$-ring structure on $\KK_0(\Var/\SS)\bigl[[\Gl(n)]^{-1},n\in \mathbb{N}\bigr]$, extending the one on $\KK_0(\Var/\SS)$. In the same way $\KK_0(\Var/\SS)[\LL^{-1}]$ can be made into a $\lambda$-ring and $\sigma^n(a\LL^m)=\sigma^n(a)\LL^{mn}$ holds for any $m\in \mathbb{Z},n\in \mathbb{N}$ and $a\in \KK_0(\Var/\SS)[\LL^{-1}]$.
\item If $(\SS,\mu,\epsilon)$ is a commutative monoid in the category of Artin stacks locally of finite type over $k$, we can define a (not special) $\lambda$-ring structure on $\KK_0(\Sta/\SS)$ resp.\ $\KK_0(\St/\SS)$ by means of
\[ \sigma^n([\XX\xrightarrow{f} \SS])=\big[\XX^n/S_n \xrightarrow{f^n/S_n} \SS^n/S_n \xrightarrow{\mu} \SS \big], \]
where $\XX^n/S_n=:\Symm^n \XX$ denotes the quotient stack of $\XX^n$ by the group $S_n$ of permutations of $n$ elements and similarly for $\SS^n/S_n$. See \cite{Ekedahl2} and \cite{Romagny1} for the precise definition of that quotient stack and \cite{Ekedahl2} for the existence of the $\lambda$-ring structure. Notice that if $\SS$ is an algebraic space, $\Sym^n \SS$ is the coarse moduli space of $\Symm^n \SS=\SS^n/S_n$.\footnote{This doesn't make sense if $\SS$ is a stack with nontrivial stabilizers. This is why we restricted ourselves to algebraic spaces in the example before.} 
\end{enumerate}
\end{Example}

The bridge between the third and the fourth Example is given as follows. To any variety $X\rightarrow \SS$ one can associate a constructible $\mathbb{Z}$-valued function on $\SS$ given by fiberwise integration with respect to the Euler characteristic with compact support, i.e.\ by pushing forward the constant function $1_X$ to $\SS$. One can show that this defines a $\lambda$-ring homomorphism from $\KK_0(\Var/\SS)$ to the constructible $\mathbb{Z}$-valued functions on $\SS$.\\
There is also a connection between the second and the fourth Example. Let for simplicity $k= \mathbb{C}$ and $X$ be a smooth quasi-projective variety over $k$. Consider the Betti cohomology functors $\Ho_c^n(-,\QQ):(\Var^{sm,qp}/k)^{op}\longrightarrow \Vect_\QQ^{f.d.}$ mapping $X$ to $\Ho_c^n(X,\mathbb{Q})$\footnote{With $\Ho_c^n(X,\QQ)$ we mean the $n$-th cohomology with compact support of the complex analytic space underlying $X$. Note that (contravariant) functoriality is only given for proper morphisms. However, using Poincar\'{e} duality we can define $f_\ast:\Ho_c^{\dim_{\mathbb{R}} X-n}(X,\QQ) \longrightarrow \Ho_c^{\dim_{\mathbb{R}} Y-n}(Y,\QQ)$ for any $f:X\longrightarrow Y$.}. They satisfy the following properties,
\begin{itemize}
 \item[(A)] $\Ho_c^n(X\times Y,\QQ) \cong \oplus_{p+q=n} \Ho_c^p(X,\QQ)\otimes \Ho_c^q(Y,\QQ)$ naturally in $X$ and $Y$,
 \item[(B)] For a closed inclusion $i:Z\hookrightarrow X$ with open complement $j:U\hookrightarrow X$ there are natural transformations $\partial^n_{Z,X}:\Ho_c^n(Z,\QQ) \longrightarrow \Ho_c^{n+1}(U,\QQ)$ such that the following long sequence is exact 
 \[ \ldots \rightarrow \Ho_c^n(U,\QQ) \xrightarrow{j_\ast} \Ho_c^n(X,\QQ) \xrightarrow{i^\ast} \Ho_c^n(Z,\QQ) \xrightarrow{\partial^n_{Z,X}} \Ho_c^{n+1}(U,\QQ) \rightarrow \ldots. \]
\end{itemize}
Let $G^{mot}$ be the group of $\mathbb{Q}$-linear automorphisms $(\tau^n: \Ho_c^n(-,\QQ) \rightarrow \Ho_c^n(-,\QQ))_{n\ge 0}$ of the functors $\Ho_c^n(-,\QQ)$ respecting properties (A) and (B). By construction there is a functor $\Var^{sm,qp}/k \longrightarrow \DER^b(G^{mot}-\Rep^{f.d.}_\QQ)$ mapping $X$ to the trivial complex $\oplus_{n\ge 0}\Ho_c^n(X,\QQ)[-n]$ of finite dimensional $\QQ$-linear $G^{mot}$-representations. Because of (B) and a result of Bittner (cf.\ \cite[Theorem 3.1]{Bittner04}) we get a natural homomorphism
\[ \KK_0(\Var/k)\cong \KK_0(\Var^{sm,qp}/k) \longrightarrow \KK_0(\DER^b(G^{mot}-\Rep_\QQ^{f.d.}))\cong \KK_0(G^{mot}-\Rep_\QQ^{f.d.}) \]
which is actually a $\lambda$-ring homomorphism as $\Ho_c^\ast(\Sym^n X,\QQ)\cong \Sym^n \Ho_c^\ast(X,\QQ)$ naturally in $X$.\\

The construction of the $\lambda$-structure on $\KK_0(\Var/\SS)$ resp.\ $\KK_0(\St/\SS)$ can also be done 
in the $G$-equivariant case. 
Ekedahl's proof that the subgroup generated by (\ref{relation}) is a $\lambda$-ideal (see \cite[Proposition 2.5]{Ekedahl2}) also generalizes literally to the equivariant case. Thus, the $\lambda$-structure on $\KK_0^G(\St/\SS)$ descends to $\KK^G(\St/\SS)$. Moreover, any homomorphism $\SS \rightarrow \TT$ of finite type defines a $\lambda$-ring homomorphism by pushing forward motives. The following proposition provides the bridge between Example (4) and (5).
\begin{Proposition} \label{Kaiso3} Let $(\SS,\mu,\epsilon)$ be a $G$-equivariant monoidal algebraic space locally of finite type with $\mu$  of finite type as before. By passing from $\KK_0^G(\ldots/\SS)$ to $\KK^G(\ldots/\SS)$ we get a sequence of $\lambda$-ring homomorphisms similar to the sequence (\ref{kette}).
\begin{eqnarray*} &&\KK^G(\Var/k) \xrightarrow{\epsilon_\ast}  \KK^G(\Var/\SS) \longrightarrow \KK^G(\Sch/\SS) \longrightarrow  \\
 && \qquad \qquad \qquad \longrightarrow \KK^G(\Sp/\SS) \longrightarrow \KK^G(\Sta/\SS) \longrightarrow \KK^G(\St/\SS). \nonumber 
\end{eqnarray*}
Moreover, the ring isomorphism of Proposition (\ref{Kaiso2}) is an isomorphism of $\lambda$-rings.
\end{Proposition}
\begin{proof}
 The arguments in the proof of Corollary 4.4 in \cite{Ekedahl2} go literally through in the $\GG_m$-equivariant case. The crucial observation is $[1/S_n]=1$ in $\KK^G(\Sta/k)$ for any $n\ge 1$ (see \cite[Theorem 4.3]{Ekedahl2}). 
\end{proof}

\begin{Definition}
 A filtered $\lambda$-ring is a $\lambda$-ring $(R,\sigma)$ together with a descending filtration $\ldots \subset F^1R \subset F^0R=R$ such that
 \begin{itemize}
  \item[(i)] $F^mR\cdot F^nR \subset F^{m+n}R$,
  \item[(ii)] $\sigma^m(F^nR)\subset F^{mn}R$.
 \end{itemize}
If $R$ with the induced topology is complete\footnote{so that in particular $\cap_{n\ge 0}F^nR=0$} $(R,\sigma,F^\cdot R)$ is called a complete filtered $\lambda$-ring. In that case we define the operation
\[ \Sym:R':= F^1R \longrightarrow R, a\longmapsto \sigma_a(1)=\sum_{n\ge 0} \sigma^n(a).\]
\end{Definition}

\begin{Lemma}
If $(R,\sigma,F^\cdot R)$ is a complete filtered $\lambda$-ring, $\Sym:R'\rightarrow R$ is an isomorphism from the additive group $(R',+)$ onto the multiplicative subgroup $1+R'$ of the group of units $(R^\times,\cdot)$. 
\end{Lemma}
\begin{proof}
By the properties of $\sigma$ we only need to show that $\Sym:R'\rightarrow 1+R'$ is an isomorphism. Indeed, any $0\neq a\in R'$ has a unique decomposition $a=\sum_{n\ge r} a_n$ with $a_n\in F^nR\setminus F^{n+1}R$, $r\ge 1, a_r\neq 0$. Then $\Sym(a)=1+a_r \mod F^{r+1}R$ is not one and injectivity follows since $\Sym$ is a group homomorphism. For surjectivity we consider $1+\sum_{n\ge 1}b_n$ with $b_n\in F^nR\setminus F^{n+1}R$. To compute $a=\sum_{n\ge 1}a_n$ as above with 
\[ \Sym(a)=1+\sum_{n\ge 1}\sum_{{1\le i_1, \ldots, i_n \atop i_1 + 2i_2 + \ldots + ni_n=n}}\sigma^{i_1}(a_1)\cdots \sigma^{i_n}(a_n)= 1 + \sum_{n\ge 1} b_n=b \] 
we have to solve the system of equations $\sum_{{1\le i_1, \ldots, i_n \atop i_1 + 2i_2 + \ldots + ni_n=n}}\sigma^{i_1}(a_1)\cdots \sigma^{i_n}(a_n)=b_n$ which can be done recursively as $\sigma^1=id$.   
\end{proof}

\begin{Remark} If $R=R\otimes \mathbb{Q}$, the statement of the lemma would also be true for the operation $\exp:R'\ni a\longmapsto \sum_{n\ge 0} \frac{a^n}{n!}\in 1 + R'$. Moreover, due to the similarity between\footnote{Notice that additional relations of the form $[\XX^n/S_n]=[\XX^n]/[S_n]$ would give $1=0$.}   $\sigma^n([\XX])=[\XX^n/S_n]$ and $[\XX]^n/[S_n]=[\XX]^n/n!$, some authors prefer to write  ``Exp'' or ``EXP'' instead of $\Sym$. We prefer the notation $\Sym$ because it is often induced by something categorical (see the previous Example (2)). In addition to this, ``Exp'' sometimes suggests relations in the context of power structures which do not hold. 
\end{Remark}

\begin{Example}
Let $\SS$ be a monoidal Artin stack as before with a descending filtration $\ldots \subset F^1\SS \subset F^0\SS=\SS$ such that $\mu(F^m\SS\times F^n\SS)\subset F^{m+n}\SS$ and $\cap_{n\in \mathbb{N}} F^n\SS=\emptyset$. Then $\KK_0(\Var/\SS)$ is a filtered $\lambda$-ring with $F^n\KK_0(\Var/\SS)=\KK_0(\Var/F^n\SS)$ and similarly for $\hat{\KK}_0(\Var/\SS)$. Moreover, $\hat{\KK}_0(\Var/\SS)$ is complete. 
One can replace $\Var/\SS$ by any other subcategory of $\St/\SS$ and similarly $\KK_0(\ldots/\SS)$ by $\KK^G(\ldots/\SS)$. 
\end{Example}

\begin{Remark} \label{relative} Everything that has been said so far can be generalized to a relative version. For this let $\Base$ be an Artin stack locally of finite type over $k$ and $\SS$ an Artin stack locally of finite type over $\Base$. If we replace any occurrence of $k$ resp.\ $\Spec(k)$ by $\Base$ we get the relative version. In particular, if $\SS$ is a commutative monoid in the category of Artin stacks locally of finite type over $\Base$, the product and the $\lambda$-ring structure on $\KK_0(\Sta/\SS)$ are given by suitable extensions of
\begin{eqnarray*}
& & [\XX\xrightarrow{f} \SS]\cdot[\YY\xrightarrow{g} \SS] \;=\; [\XX\times_\Base \YY \xrightarrow{f\times_\Base g} \SS \times_\Base \SS \xrightarrow{\;\mu\;} \SS] \\
& & \sigma^n([\XX\xrightarrow{f} \SS]) \;=\; [ \underbrace{\XX \times_\Base \ldots \times_\Base \XX}_{n\;\mbox{\scriptsize times}} /S_n \xrightarrow{f^n/S_n} \underbrace{\SS \times_\Base \ldots \times_\Base \SS}_{n\;\mbox{\scriptsize times}} /S_n \xrightarrow{\;\mu\;} \SS]  
\end{eqnarray*}
with unit $[\Base \xrightarrow{\epsilon} \SS]$. In the special case $\SS=\Base$ we get the so-called fiber product on $\KK_0(\Sta/\SS)$. If $\SS$ and $\Base$ are algebraic spaces, we can replace the quotient stacks by their coarse moduli spaces $\Sym^n_\Base(-)$ to obtain a $\lambda$-ring structure on $\KK_0(\Var/\SS)$. We will only make use of this more general structure in the proof of Theorem \ref{symvanishing} and Theorem \ref{ThomSebastiani} with $\Base=\mathbb{N}_{>0}$. The reader is invited to convince himself that any result in section 4 goes through literally if we replace $k$ with $\Base$.
\end{Remark}


\section{Vanishing cycles - the equivariant case}

Let $\SS$ be a commutative monoid over $\BB$ as before, e.g.\ $\SS=\BB$, and consider the monoid $\AA^1_\SS=\SS\times_k\AA^1_k$ locally of finite type over $\BB$ with $\GG_m$ acting fiberwise with weight $d\ge 0$, i.e.\ $g\cdot(u,z)=(u,g^dz)$ ($\AA^1$ is given the structure of a monoid via addition). For notational convenience we will assume  that $\SS$ is an algebraic space or even a variety and any quotient by some permutation group $S_n$ has to be taken without stabilizers. The stacky case is completely analogous. Denote by $\KK_0^{\GG_m,d}(\Var/\AA^1_\SS)$ the corresponding $\lambda$-ring of $\GG_m$-equivariant motives and similarly we denote by $\KK_0^{\GG_m,d}(\Var/\Gl_\SS(1))$ the subgroup supported on the complement $\SS \times_k \GG_m$ of the zero section in $\AA^1_\SS$. The latter group can canonically identified with $\KK_0^{\mu_d}(\Var/\SS)$ by taking the fiber over the unit section $1_\SS:\SS \rightarrow \Gl_\SS(1)$. The inverse operation is given by linear extension of
\[ [X \xrightarrow{f}\SS ] \longmapsto [X\times_{\mu_d} \GG_m \xrightarrow{f\times_k z^d} \SS\times_k\GG_m] \]
with $\GG_m$ acting on $X\times_{\mu_d}\GG_m$ by multiplication on the second factor. Indeed, if $f=(f_1,f_2):X \rightarrow \SS\times_k\GG_m$ is homogeneous of degree $d$, i.e.\ $f(g\cdot x)=g^df(x)=(f_1(x),g^df_2(x))\;\forall \;g\in\GG_m,x\in X$, we can pass to an \'{e}tale cover with Galois group $\mu_d$ and assume that $f$ is homogeneous of degree one. In that case $X$ is isomorphic to $f^{-1}(1)\times_k \GG_m$ with the isomorphism given by $x\longmapsto (f_2(x)^{-1}\cdot x,f_2(x))$. In the sequel we will use, however, a different identification of $\KK_0^{\GG_m,d}(\Var/\Gl_\SS(1))$ with $\KK_0^{\mu_d}(\Var/\SS)$, namely the one described above followed by multiplication with $-1$. We will justify this choice later.\\
Note that for $d|d'$ the $\lambda$-ring $\KK^{\GG_m,d}_0(\Var/\AA^1_\SS)$ maps canonically into $\KK^{\GG_m,d'}_0(\Var/\AA^1_\SS)$ by replacing the $\GG_m$-action on $X$ in $X\xrightarrow{f} \AA^1_\SS$ with the new action $g\star x:= g^{d'/d}\cdot x$. This map is compatible with the subgroups $\KK^{\GG_m,d}_0(\Var/\Gl_\SS(1))$ on which it is actually an embedding. Indeed, if we identify the latter subgroup with $\KK^{\mu_d}_0(\Var/\SS)$ as shown before, the map corresponds to the replacement of the $\mu_d$-action with the $\mu_{d'}$-action by means of the epimorphism $\mu_{d'} \twoheadrightarrow \mu_d$ sending a root to its $d'/d$-th power. A left inverse of this construction is given by modding out the action of the kernel $\mu_{d'/d} \subset \mu_{d'}$ of that epimorphism. The quotient $X/\mu_{d'/d}$ exits as our action was good by assumption. Let us introduce the following notation
\[ \KK_0^{\hat{\mu}}(\Var/\SS) := \varinjlim_d \KK^{\mu_d}_0(\Var/\SS) = \bigcup_d \KK^{\mu_d}_0(\Var/\SS) = \varinjlim_d \KK^{\GG_m,d}_0(\Var/\Gl_\SS(1)).\]
We will now equip $\KK^{\hat{\mu}}_0(\Var/\SS)$ with a new structure of a $\lambda$-ring. To do this we consider the $\lambda$-ring $\KK^{\GG_m,d}_0(\Var/\AA^1_\SS)$ together with the subgroup $\mathfrak{I}_d$ given by the image of the pull-back 
\[ \KK^{\GG_m}_0(\Var/\SS) \ni [Y\xrightarrow{g} \SS] \longmapsto [Y\times_k \AA^1_k \xrightarrow{g\times id} \SS\times_k \AA^1_k]\in \KK^{\GG_m,d}_0(\Var/\AA^1_\SS) \]
with $\GG_m$ acting on $Y\times_k \AA^1_k$  by  $g\cdot (y,z)=(g\cdot y,g^dz)$. The restriction to the zero section $\SS\hookrightarrow \AA^1_\SS$ provides us with a left inverse and we can identify the subgroup $\mathfrak{I}_d$ with $\KK_0^{\GG_m}(\Var/\SS)$. For $d|d'$ we have the commutative diagram
\[ \xymatrix {\KK_0^{\GG_m}(\Var/\SS) \cong \mathfrak{I}_d \ar@{^{(}->}[r] \ar[d] & \KK_0^{\GG_m,d}(\Var/\AA^1_\SS) \ar[d] \\ \KK_0^{\GG_m}(\Var/\SS) \cong \mathfrak{I}_{d'} \ar@{^{(}->}[r] & \KK_0^{\GG_m,d'}(\Var/\AA^1_\SS),  }  
\]
where the vertical maps replace the action of $g\in \GG_m$ with the action of $g^{d'/d}$ as seen above.
\begin{Lemma} \label{lemma}
The subgroup $\mathfrak{I}_d \cong \KK_0^{\GG_m}(\Var/\SS)$  is a $\lambda$-ideal of $\KK^{\GG_m,d}_0(\Var/\AA^1_\SS)$ and if we identify the quotient $\lambda$-ring $\KK^{\GG_m,d}_0(\Var/\AA^1_\SS) / \mathfrak{I}_d$ with the complement $\KK^{\GG_m,d}_0(\Var/\Gl_\SS(1))\cong \KK^{\mu_d}_0(\Var/\SS)$ of $\mathfrak{I}_d$ in $\KK^{\GG_m,d}_0(\Var/\AA^1_\SS)$, the latter gets a natural structure of a $\lambda$-ring. Moreover, we can pass to the inductive limit over $d$ and get a $\lambda$-ring structure on $\KK^{\hat{\mu}}_0(\Var/\SS)$.
\end{Lemma}
\begin{proof}
It is not difficult to see that $\KK_0^{\GG_m,d}(\Var/\Gl_\SS(1))$ is indeed a complement of $\mathfrak{I}_d$ in $\KK_0^{\GG_m,d}(\Var/\AA^1_\SS)$ and the projection onto this complement is given by linear extension of
\[ [ X\xrightarrow{\;f\;} \AA^1_\SS] \longmapsto \Big([f^{-1}(\Gl_\SS(1)) \xrightarrow{\;f\;} \Gl_\SS(1)] - [f^{-1}(0)\times_k \GG_m \xrightarrow{f\times id} \SS\times_k\GG_m]\Big) \]
respectively of 
\[ [ X\xrightarrow{\;f\;} \AA^1_\SS] \longmapsto \Bigr([f^{-1}(0) \xrightarrow{p_\SS\circ f\;} \SS] - [f^{-1}(1) \xrightarrow{p_\SS\circ f} \SS]\Bigl) \]
after identifying $\KK_0^{\GG_m,d}(\Var/\Gl_\SS(1))$ with $\KK_0^{\mu_d}(\Var/\SS)$. Here, $\mu_d\subset \GG_m$ acts both on $f^{-1}(0)$ and on $f^{-1}(1)$. It remains to show that $\mathfrak{I}_d$ is a $\lambda$-ideal. This can be done by looking at generators $[X\xrightarrow{f} \AA^1_\SS]$ and $[Y\times_\SS \AA^1_\SS \xrightarrow{p_{\AA^1_\SS}} \AA^1_\SS]$  of $\KK^{\GG_m,d}_0(\Var/\AA^1_\SS)$ respectively $\mathfrak{I}_d$, with $[Y\xrightarrow{g} \SS]\in \KK_0^{\GG_m}(\Var/\SS)$. For the product 
\[ [X\xrightarrow{f} \AA^1_\SS] \cdot [Y\times_\SS \AA^1_\SS \xrightarrow{p_{\AA^1_\SS}} \AA^1_\SS] = [X\times_\BB (Y \times_\SS \AA^1_\SS) \xrightarrow{f p_X + p_{\AA^1_\SS}} \AA^1_\SS] \]
we mention 
\[ [X\times_\BB (Y \times_\SS \AA^1_\SS) \xrightarrow{f p_X + p_{\AA^1_{\SS}}} \AA^1_\SS] = [(X\times_\BB Y) \times_\SS \AA^1_\SS \xrightarrow{p_{\AA^1_\SS}} \AA^1_\SS] \in \mathfrak{I}_d\] 
which is obtained by composing $p_{\AA^1_\SS}$ with the $\GG_m$-equivariant isomorphism 
\[ X\times_\BB (Y \times_\SS \AA^1_\SS) \xrightarrow{(p_X, p_Y, f p_X + p_{\AA^1_\SS})} (X\times_\BB Y) \times_\SS \AA^1_\SS \]
as $g\in \GG_m$ acts by multiplication with $g^d$ on $\AA^1_\SS$. Similarly, we get for all $n>0$ 
\begin{eqnarray*} && [\Sym^n_\BB(Y \times_\SS \AA^1_\SS) \xrightarrow{\Sym^n_\BB p_{\AA^1_\SS}} \Sym^n_\BB \AA^1_\SS \xrightarrow{\;+\;} \AA^1_\SS ] \\
&=& [ (Y^n \times_\BB \AA^n_\BB)/S_n \xrightarrow{(g^n\times id_{ \AA^n_\BB})/S_n}  (\SS^n \times_\BB \AA^n_\BB)/S_n \xrightarrow{\mu\times+} \SS\times_\BB\AA^1_\BB \cong \AA^1_\SS] \\
& = & [ ( Y^n \times_\BB \AA^n_\BB)^0/S_n\times_\BB \AA^1_\BB \xrightarrow{\bigl(g^n \times id_{\AA^{n-1}_\BB}\bigr)/S_n \times id_{\AA^1_\BB}}  (\SS^n \times_\BB \AA^n_\BB)^0/S_n\times_\BB \AA^1_\BB  \\ 
& & \hspace{6.5cm}\xrightarrow{\mu\circ p_{ \SS^n} \times id_{\AA^1_\BB}} \SS\times_\BB \AA^1_\BB \cong \AA^1_\SS] \\
& = & [( Y^n \times_\BB \AA^n_\BB)^0/S_n \times_\SS \AA^1_\SS \xrightarrow{p_{\AA^1_\SS}} \AA^1_\SS]\in \mathfrak{I}_d, 
\end{eqnarray*}
where we used the notation 
\[ (Y^n\times_\BB \AA^n_\BB)^0:=\{(y_1,\ldots,y_n,z_1,\ldots,z_n) \in Y^n\times_\BB \AA^n_\BB \mid z_1 + \ldots z_n=0 \} \] and similarly for $(\SS^n\times_\BB \AA^n_\BB)^0$ as well as the obvious  $\GG_m\times S_n$-equivariant isomorphism $(Y^n\times_\BB \AA^n_\BB)^0\times_\BB \AA^1_\BB \longrightarrow Y^n \times_\BB \AA^n_\BB$ sending $((y_1,\ldots,y_n,z_1,\ldots,z_n),z)$ to $(y_1,\ldots,y_n,z_1 + z/n,\ldots, z_n + z/n)$.\\
In particular, for any $a\in \KK_0^{\GG_m,d}(\Var/\AA^1_\SS)$ and any $b\in \mathfrak{I}_d$ one has $\sigma^p(a+b)=\sum_{n=0}^p \sigma^{p-n}(a)\sigma^{n}(b)\equiv\sigma^p(a) \mod \mathfrak{I}_d$ and the residue class of $\sigma^p(a)$ depends only on the residue class of $a$ in $\KK_0^{\GG_m,d}(\Var/\AA^1_\SS)/\mathfrak{I}_d$. \\
Neither the product nor the $\lambda$-ring structure on $\KK_0^{\GG_m,d}(\Var/\AA^1_\SS)$ depend on the weight $d$. Thus, $\KK^{\GG_m,d}_0(\Var/\AA^1_\SS) \longrightarrow \KK^{\GG_m,d'}_0(\Var/\AA^1_\SS)$ is a $\lambda$-ring homomorphism and the same holds for the quotients $\KK^{\GG_m,d}_0(\Var/\AA^1_\SS)/\mathfrak{I}_d \longrightarrow \KK^{\GG_m,d'}_0(\Var/\AA^1_\SS)/\mathfrak{I}_{d'}$. The statement for the inductive limit follows. 
\end{proof}

For the sake of completeness we will also give explicit formulas for the product and the $\lambda$-ring structure in terms of $\KK_0^{\hat{\mu}}(\Var/\SS)$. For natural numbers $d,n>0$ we introduce the following notation
\begin{eqnarray*} 
 J_{n,0}^d&:=&\{ (z_1,\ldots,z_n)\in \GG_m^n \mid z_1^d + \ldots + z_n^d=0 \} \\
 J_{n,1}^d&:=&\{ (z_1,\ldots,z_n)\in \GG_m^n \mid z_1^d + \ldots + z_n^d = 1 \}.
\end{eqnarray*}
The groups $\mu_d^n$ and $S_n$ obviously act on $J^d_{n,0}$ and $J^d_{n,1}$. Moreover, the action of the diagonal subgroup $\mu_d\hookrightarrow \mu_d^n$ commutes with the action of $\mu_d^n$ and of $S_n$. If $[X\rightarrow \SS]$ and $[Y\rightarrow \SS]$ are two elements of $\KK_0^{\hat{\mu}}(\Var/\SS)$ with $\hat{\mu}$ acting on $X$ and $Y$ by means of $\mu_d$, their product is given by 
\[ [X\rightarrow \SS]\cdot [Y\rightarrow \SS]= [(X\times_\BB Y)\times_{\mu_d^2}J^d_{2,0} \longrightarrow \SS] -  [(X\times_\BB Y)\times_{\mu_d^2}J^d_{2,1} \longrightarrow \SS]. \]
with $\hat{\mu}$ acting on $J^d_{2,0}$ resp.\ $J^d_{2,1}$ via $\mu_d\hookrightarrow \mu_d^2$. Similarly, 
\[ \sigma^n(-[X\rightarrow \SS])= [(X^n\times_{\mu_d^n}J^d_{n,0})/S_n\longrightarrow \SS] - [(X^n\times_{\mu_d^n}J^d_{n,1})/S_n\longrightarrow \SS] \]
with $\hat{\mu}$ acting on $J^d_{n,0}$ resp.\ $J^d_{n,1}$ via $\mu_d \hookrightarrow \mu_d^n$.  \\
Notice that we used our sign convention when identifying $\KK_0^{\GG_m,d}(\Var/\Gl_\SS(1))$ with $\KK^{\mu_d}_0(\Var/\SS)$. The reason is the following. The $\lambda$-ring $\KK_0^{\GG_m}(\Var/\SS)$ maps as a $\lambda$-ring to $\KK_0^{\GG_m,d}(\Var/\AA^1_\SS)$ by means of the push forward along the homomorphism $\SS \xrightarrow{\,0\,} \AA^1_\SS$ of monoids as seen before. 
By composing it with the quotient map $\KK_0^{\GG_m,d}(\Var/\AA^1_\SS) \longrightarrow \KK_0^{\GG_m,d}(\Var/\Gl_\SS(1))\cong \KK_0^{\mu_d}(\Var/\SS)$ we get a natural $\lambda$-ring homomorphism $\KK_0^{\GG_m}(\Var/\SS) \longrightarrow \KK_0^{\mu_d}(\Var/\SS)$ which by our sign convention is just the obvious map considering any variety with good $\GG_m$-action as a variety with good $\mu_d$-action by restricting the action to $\mu_d\subset \GG_m$. In particular, our sign convention guarantees that the $\lambda$-ring structure on $\KK_0^{\hat{\mu}}(\Var/\SS)$ is just an ``extension'' of the one on $\KK_0(\Var/\SS) \longrightarrow \KK_0^{\GG_m}(\Var/\SS)$. 

\begin{Remark}
The ring $\KK_0^{\muu}(\Var/\SS)$ and its quotient $\KK^{\muu}(\Var/\SS)$ has been introduced in \cite{GuibertLoeserMerle} to prove a kind of Thom-Sebiastiani theorem. The product was already defined by Looijenga \cite{Looijenga1} in a different way. This section and, in particular, Lemma \ref{lemma} was inspired by the paper \cite{KS1} of Kontsevich and Soibelman which simplified the ad hoc construction of Looijenga a lot. They show that $\KK_0(\Var/\SS) \hookrightarrow \KK_0(\Var/\AA^1_\SS)$ is an ideal. Hence, they get a ring structure on $\KK_0(\Var/\Gl_\SS(1))$ and, by forgetting the $\GG_m$-action, we obtain a ($\lambda$-)ring homomorphism $\KK_0^{\muu}(\Var/\SS)\longrightarrow \KK_0(\Var/\Gl_\SS(1))$. 
\end{Remark}

As mentioned before, everything that has been said so far generalizes easily to $\KK_0^{\hat{\mu}}(\St/\SS)$ and $\KK_0^{\hat{\mu}}(\Sta/\SS)$. Moreover, the construction of the $\lambda$-ring structure is compatible with relation (\ref{relation}) in section 2 allowing us to put a new $\lambda$-ring structure on for example $\KK^{\hat{\mu}}(\Var/\SS)$. By Remark \ref{Kaiso4} the $\lambda$-ideal of $\KK^{\GG_m,d}(\Var/\AA^1_\SS)$ analoguos to $\mathfrak{I}_d$ is isomorphic to $\KK(\Var/\SS)$ and we can replace the $\hat{\mu}$-actions on $[(X^n\times_{\mu_d^n}J^d_{n,0})/S_n\longrightarrow \SS]$ and on $[(X\times_\BB Y)\times_{\mu_d^2}J^d_{2,0} \longrightarrow \SS]$ with the trivial one. Using relation (\ref{relation}) it is also not difficult to see that $[(X\times_\BB Y)\times_{\mu_d^2}J^d_{2,0} \longrightarrow \SS]$ coincides with $(\LL-1)[X\times_{\mu_d} Y \longrightarrow \SS]$ in $\KK^{\muu}(\Var/\SS)$ (cf.\ \cite[section 7]{Looijenga1}).

\begin{Example}
The residue class of the map $[X=\AA^1_\BB \xrightarrow{f=z^2} \AA^1_\BB]\in \KK^{\GG_m,2}(\Var/\AA^1_\BB)$ in $\KK^{\hat{\mu}}(\Var/\BB)$ is given by $[f^{-1}(0)]-[f^{-1}(1)]=1- [\mu_2]$ with $\mu_2$ carrying the obvious $\hat{\mu}$-action. Similarly, the residue class of $[\AA^2_\BB \xrightarrow{g=z_1^2+z_2^2} \AA^1_\BB]$ can be easily computed by coordinate change to be $[g^{-1}(0)] - [g^{-1}(1)]=\LL$ because the class of the variety $\Gl_\BB(1)=\GG_m$ with $\mu_2$-action given by sign change is the same as the class of $\LL-1$ by (2) and the scissor relation. On the other hand, the latter residue class is by definition of the convolution product and Lemma (\ref{lemma}) just the square of $\LL^{1/2}:=1-[\mu_2]$. We will see later that $-\LL^{1/2}=[\mu_2]-1$ is a line element.   
\end{Example}
Generalizing the previous example, the residue class in $\KK^{\muu}(\Var/k)$ of a nondegenerate quadratic function $f:\AA^n_k \rightarrow \AA^1_k$ is given by $\LL^{n/2}$. However, the critical scheme $\Crit(f)\subset \AA^n_k$ is always $\Spec(k)$ independent of $n$. To get some measure of $\Spec(k)$ independent of its representation as a critical locus, we should normalize the residue class motivating the following definition.

\begin{Definition}
Let $X$ be an equidimensional variety locally of finite type over $\SS$. Assume that $X$ carries a good $\GG_m$-action. Let $X^s\subset X$ be a $\GG_m$-invariant locally closed subvariety of finite type over $\SS$. If $f:X \rightarrow \AA^1_\SS$ is a $\GG_m$-equivariant map over $\SS$ with $\GG_m$ acting on $\AA^1_\SS$ with some positive weight $d$, we denote by $\int_{X^s}\phi^{eq}_f\in \KK^{\hat{\mu}}(\Var/\SS)[\LL^{-1/2}]$ the rescaled residue class of $[X^s \xrightarrow{\;f|_{X^s}\;} \AA^1_\SS]$, i.e.\ 
\[ \int_{X^s}\phi^{eq}_f:= \LL^{-\frac{\dim X}{2}} \big( [f^{-1}(0)\cap X^s \xrightarrow{\;pr_{\mathfrak{M}}\circ f\;} \SS] - [f^{-1}(1)\cap X^s \xrightarrow{\;pr_{\mathfrak{M}}\circ f\;} \SS] \big) \]
with $\hat{\mu}$ acting nontrivially only on $f^{-1}(1)\cap X^s$ via $\mu_d$. Moreover, if $\XX^s\subset \XX$ are Artin stacks with affine stabilizers satisfying the same properties, we define $\int_{\XX^s}\phi^{eq}_f\in \KK^{\muu}(\Sta/\SS)$ in the same way.
\end{Definition}

As an immediate consequence of Lemma \ref{lemma} and the definition we obtain the following proposition.
\begin{Proposition} \label{basicprop}
Let $\XX,\YY$ and $\ZZ$ be equidimensional Artin stacks locally of finite type over $\SS$ with affine stabilizers. Assume, moreover, that $\XX,\YY$ and $\ZZ$ are equipped with a good $\GG_m$-action and that $\XX^s,\YY^s,\ZZ^s$ are locally closed $\GG_m$-invariant substacks in $\XX,\YY$ resp.\ $\ZZ$. Let $\pi:\ZZ \rightarrow \XX$, $f:\XX \rightarrow \AA^1_\SS$ and $g:\YY \rightarrow \AA^1_\SS$ be $\GG_m$-equivariant maps with respect to suitable actions of $\GG_m$ on $\AA^1_\SS$ of positive weight. By choosing the $\GG_m$-actions correspondingly, we can assume that $f+g:\XX\times_\BB \YY \xrightarrow{f\times g}  \AA^1_\SS \times_\BB \AA^1_\SS \xrightarrow{\:+\:} \AA^1_\SS$ and $\SSymm_+^n(f):\Symm^n_\BB \XX \xrightarrow{\Symm^n(f)} \Symm^n_\BB \AA^1_\SS \xrightarrow{\:+\:} \AA^1_\SS$ are also $\GG_m$-equivariant. Then
\begin{enumerate}
\item $\int_{\XX^s}\phi^{eq}_f=\LL^{-\frac{\dim\XX}{2}} [\XX^s \rightarrow \SS]$ if $f\equiv 0$,
\item $\int_{\XX^s\sqcup \XX^s_2} \phi^{eq}_f=\int_{\XX^s}\phi^{eq}_f + \int_{\XX^s_2} \phi^{eq}_f$ for any $\GG_m$-invariant locally closed substack $\XX^s_2\subset \XX$ not intersecting $\XX^s$,
\item $\int_{\XX^s\times \YY^s}\phi^{eq}_{f+g}=\int_{\XX^s}\phi^{eq}_f \cdot \int_{\YY^s}\phi^{eq}_g$ (Thom--Sebastiani),
\item $\LL^{n\dim \XX/2}\int_{\Symm^n \XX^s} \phi^{eq}_{\SSymm_+^n(f)}=\sigma^n( \LL^{\dim  \XX/2}\int_{\XX^s}\phi^{eq}_f)$ for all $n \ge 0$ and a similar statement holds for $\Sym^n(f):\Sym^n \XX \longrightarrow \AA^1_\SS$ if $\XX$ and $\SS$ are algebraic spaces,
\item $\int_{\ZZ^s}\phi^{eq}_{f\circ \pi}= \LL^{-\frac{r}{2}}[\mathcal{F}]\int_{\XX^s}\phi^{eq}_f$ if $\pi|_{\ZZ^s}:\ZZ^s\rightarrow \XX^s$ is a Zariski locally trivial fibration with $r$-dimensional fiber\footnote{If the fiber $F$ carries a nontrivial $\GG_m$-action, its class in $\KK^{\muu}(\Var/k)[\LL^{-1/2}]$ is given by the residue class of $[F\rightarrow 0]\in \KK^{\GG_m,1}(\Var/\AA^1_k)$.} $\mathcal{F}$. The same holds if $\pi|_{\ZZ^s}$ is a vector bundle or a $\Gl_k(n)$-principal bundle.
\end{enumerate}
\end{Proposition}
\begin{proof}
Properties (1),(2),(3),(4) and (5) follow directly from the definition and Lemma \ref{lemma}. If $\pi|_{\ZZ^s}$ is a vector or a principal bundle, we use either the relation (\ref{relation}) or Lemma \ref{principal}. 
\end{proof}

\begin{Example}
Let $f:\AA^n_\BB \rightarrow \AA^1_\BB$ be a quadratic form of rank $r$. Since $\bar{k}=k$,  $\int_{\AA^n_\BB}\phi^{eq}_f=\LL^{\frac{n-r}{2}}$ by equation (1) and (3) of Prop.\ \ref{basicprop}. Let us now consider the case $n=r=1$. We use the identification $\Sym^q \AA^1_k\cong \AA^q_k$ given by 
\[ \AA^q_k\ni(z_1,\ldots,z_q)\longmapsto (z_1+\ldots+z_q, \ldots, z_1^q+\ldots+ z_q^q)\in \AA_k^q, \]
such that $\GG_m$ acts with weights $1,2,\ldots,q$ on $\AA^q_k$. Hence, the function $\SSym_+^q(f):\AA^q_k \ni(z_1,\ldots, z_q)\longmapsto z_2\in \AA^1_k$ is in the $\lambda$-ideal $\mathfrak{I}_2$ for $q>1$ and $\int_{\AA^q_k}\phi^{eq}_{\SSym_+^q(f)}=0$ follows. By Proposition \ref{basicprop} (4) we obtain $\sigma_{1-[\mu_2]}(t) = 1 + (1-[\mu_2])t$ and, thus, $\sigma_{[\mu_2]-1}(t) = 1 / (1 - ([\mu_2]-1)t )$, in other words, $[\mu_2]-1$ is a line element in $\KK^{\muu}(\Var/k)$ while $1-[\mu_2]$ is not. Notice that $\KK^{\muu}(\Var/k)[\LL^{-1}]=\KK^{\muu}(\Var/k)[\LL^{-1/2}]$ and $\sigma^n(a(-\LL^{1/2})^m)=\sigma^n(a)(-\LL^{1/2})^{mn}$ for all $m\in\mathbb{Z},n\in \mathbb{N}$ and $a\in\KK^{\muu}(\Var/k)[\LL^{-1/2}]$.
\end{Example}

Let us finally show that the Euler characteristic (with compact support) provides a $\lambda$-ring homomorphism $\KK^{\muu}(\Var/k)\longrightarrow \mathbb{Z}$ with respect to the exotic $\lambda$-ring structure on $\KK^{\muu}(\Var/k)$ introduced in this section. Consider the group $\CF^{\GG_m}(\AA^1_k,\mathbb{Z})$ of $\GG_m$-invariant $\mathbb{Z}$-valued constructible functions $f$ on $\AA^1_k$. Such a function is determined by its values on $0$ and $1$. As in Example \ref{lambdarings} (3) we equip this group with a $\lambda$-ring structure. Note that the unit is given by the characteristic function $\delta_0$ of $\{0\}\subset \AA^1_k$. Similarly to Lemma \ref{lemma} one shows that the subgroup of constant functions is a $\lambda$-ideal. The quotient group can be identified with $\mathbb{Z}$ by taking the difference $f(0)-f(1)$. Thus, $\mathbb{Z}$ gets an induced $\lambda$-ring structure with unit $1=\delta_0(0)-\delta_0(1)$. By the basic properties of $\lambda$-rings there is a unique $\lambda$-ring structure on $\mathbb{Z}$ with unit $1$ given by Example \ref{lambdarings} (1). Hence, we obtain a $\lambda$-ring homomorphism $\CF^{\GG_m}(\AA^1_k,\mathbb{Z}) \longrightarrow \mathbb{Z}$ vanishing on constant functions.\\ 
For any $d>0$ there is obviously a $\lambda$-ring homomorphism $\KK^{\GG_m,d}(\Var/\AA^1_k)\longrightarrow \CF^{\GG_m}(\AA^1_k,\mathbb{Z})$ by taking the Euler characteristic fiberwise. As $\mathfrak{I}_d$ maps to the subgroup of constant functions, we get a $\lambda$-ring homomorphism $\chi:\KK^{\mu_d}(\Var/k)\longrightarrow \mathbb{Z}$ by passing to the quotients. Taking our identifications into account, $\chi$ maps a variety with good $\mu_d$-action to its Euler characteristic. This is independent of $d$ and we can finally pass to the limit $d\to \infty$.  \\ 
This construction can be extended to get $\lambda$-ring homomorphisms to Grothendieck groups of $\muu$-equivariant  resp.\ fractional mixed Hodge structures. See \cite{DavisonMeinhardt2} for more details.

\section{Vanishing cycles - the general case}

Let $X$ be a smooth $k$-variety of finite type and dimension $d$ and let $f:X\rightarrow \AA_k^1$ be a regular map. Denote by $X_0$ the (reduced) fiber over $0\in \AA_k^1$ and by $\Crit(f)$ the critical locus of $f$. The latter is contained in finitely many fibers of $f$ and by shrinking $X$ and replacing $f$ with $f-c$ for $c\in \AA_k^1$ constant, we can assume $\Crit(f)\subset X_0$. 
Note that for $k=\mathbb{C}$ the cohomology of the classical nearby cycle sheaf on $X_0$ is more or less by definition the cohomology of a nearby fiber $f^{-1}(\varepsilon)$ with $0\neq \varepsilon\in \AA^1_k$ very small. A naive motivic replacement could be the motive of the fiber $f^{-1}(\varepsilon)$. But this motive depends on $\varepsilon$ and there is in general no way to consider $f^{-1}(\varepsilon)$ as a motive on $X_0$ generalizing the classical nearby cycle sheaf. However, there is one exception. For $\GG_m$-equivariant $f$ the motive of $f^{-1}(\varepsilon)$ is independent of $\varepsilon\neq 0$ and this was our motivation for the (integral over the) naive vanishing cycle $\int_{X}\phi^{eq}_f$ defined in the previous section for $\GG_m$-equivariant functions. In the general case Denef and Loeser \cite{DenefLoeser2},\cite{DenefLoeser1} constructed a motivic version of the sheaf of vanishing cycles associated to $f$, i.e.\ a motive $\phi_f$ in $\KK^{\muu}(\Var/X_0)[\LL^{-1/2}]$, by considering certain arc-spaces. We will give a slightly different approach to it using various $\GG_m$-equivariant functions associated to $f$ and their naive vanishing cycles.\\
For $X\rightarrow \AA_k^1$ and $n\ge0$ we denote by $\Lo_n(X)$ the space of arcs of length $n$ in $X$, i.e.\ the scheme representing the functor $Y\longmapsto \Hom_{\Sch}(Y\times_k \Spec(k[t]/(t^{n+1})), X)$. By functoriality of the construction there is a morphism $\Lo_n(f):\Lo_n(X)\longrightarrow \Lo_n(\AA_k^1)\cong \AA^{n+1}_k$. Denote by $\Lo_{n,\GG_m}(X) \subset \Lo_n(X)$ the preimage of $\{0\} \times_k \ldots \times_k\{0\}\times_k \GG_m \subset \AA^{n+1}_k\cong \Lo_n(\AA^1_k)$, in other words, the subvariety of arcs $\gamma$ such that $f\circ \gamma$ is induced by $k[X]\ni P(X)\mapsto P(zt^n) \mod t^{n+1}\in k[t]/(t^{n+1})$ for some $z\in \GG_m$. \\
Let us also introduce the obvious projections $\pi^n_m:\Lo_n(X) \longrightarrow \Lo_m(X)$ for $m\le n$ and write $\pi^n$ for $\pi^n_0$. For any $\gamma \in \Lo_n(X)|_{X_0}:= (\pi^n)^{-1}(X_0)$ the arc $\Lo_n(f)(\gamma)=f\circ \gamma$ is either zero and $\gamma$ was an arc in $X_0$ or there is a first nonzero coefficient and $\gamma\in (\pi^n_m)^{-1}(\Lo_{m,\GG_m}(X))$ for some $1\le m\le n$. The latter space is an affine bundle with fiber $\AA^{(n-m)d}_k$ over $\Lo_{m,\GG_m}(X)$ as $X$ is smooth. In other words, we have the following stratification
\begin{equation} \label{stratification} \Lo_n(X)|_{X_0}=(\pi^n)^{-1}(X_0) = \Lo_n(X_0) \sqcup \sqcup_{m=1}^n (\pi^n_m)^{-1}(\Lo_{m,\GG_m}(X)). \end{equation}
The space on the left hand side is also an affine fibration with fiber $\AA^{nd}_k$ over $X_0$. Taking the motives relative to  $X_0$ yields
\[ [X_0\rightarrow X_0]=\frac{[(\pi^n)^{-1}(X_0)\rightarrow X_{0}]}{\LL^{nd}} = \frac{[\Lo_n(X_0) \rightarrow X_0]}{\LL^{nd}} + \sum_{m=1}^n \frac{[\Lo_{m,\GG_m}(X)\rightarrow X_0]}{\LL^{md}}.\] 
By multiplying this equation with the formal variable $T^n$ and summing over $n\ge 1$, we finally get the equation
\begin{eqnarray}  \frac{[X_0\rightarrow X_0]T}{1-T} &=& Z_f^{naive}(T)\frac{1}{1-T} + J_{X_0}(\LL^{-d}T)-[X_0\rightarrow X_0]  \;\mbox{ or } \nonumber \\ 
J_{X_0}(\LL^{-d}T) &=& \frac{[X_0\rightarrow X_0] - Z_f^{naive}(T)}{1-T},
\end{eqnarray}
where we used the shorthands $J_{X_0}(T)=\sum_{n\ge 0} [\Lo_n(X_0) \rightarrow X_0] T^n$ and $Z_f^{naive}(T)=\sum_{n\ge 1} [\Lo_{n,\GG_m}(X) \rightarrow X_0]\LL^{-nd}T^n$ as in \cite[Definition 3.2.1]{DenefLoeser2}.\\
We also define a map $f_n:\Lo_n(X) \longrightarrow \AA^1_k$ by the composition 
\[\Lo_n(X)\xrightarrow{\Lo_n(f)} \Lo_n(\AA^1_k)\cong \AA^{n+1}_k \xrightarrow{pr_{n+1}} \AA^1_k.\] 
There is a natural $\GG_m$-action on $\Spec(k[t]/(t^{n+1}))$ induced by $t\mapsto gt$ for $g\in \GG_m$. This induces an action on $\Lo_n(-)$ and $\Lo_n(f)$ is $\GG_m$-equivariant. Moreover, $f_n$ is homogeneous of degree $n$ if we use the standard $\GG_m$-action on $\AA^1_k$. Let us introduce the following generating series
\begin{eqnarray*} \lefteqn{Z_f^{eq}(T)\;  = \; \sum_{n\ge 1} \int_{\Lo_n(X)|_{X_0}} \phi^{eq}_{f_n} T^n } \\ & = & \sum_{n\ge 1} \LL^{-(n+1)d/2}\Bigl( \bigl[f_n^{-1}(0)\cap \Lo_n(X)|_{X_0} \rightarrow X_0\bigr] - \bigl[f_n^{-1}(1)\cap \Lo_n(X)|_{X_0} \rightarrow X_0\bigr]\Bigr)T^n 
\end{eqnarray*}
in $\hat{\KK}^{\muu}(\Var/X_0\times \mathbb{N}_{>0})$ with $\muu$ acting nontrivially only on $f_n^{-1}(1)\cap \Lo_n(X)|_{X_0}$. To compute $Z^{eq}_f(T)$ we use the stratification (\ref{stratification}) and the notation $\Lo_{n,1}(X):=\Lo_{n,\GG_m}(X)\cap f_n^{-1}(1)$. Observe that
\begin{eqnarray*} f_n^{-1}(z)\cap \Lo_n(X_0) & = & \begin{cases} \emptyset & \mbox{for }z\neq 0, \\ \Lo_n(X_0) & \mbox{for }z=0, \end{cases} \\ 
f_n^{-1}(z)\cap \Lo_{n,\GG_m}(X) & \cong & \begin{cases} \Lo_{n,1}(X)& \mbox{for } z\neq 0, \\ \emptyset & \mbox{for }z=0. \end{cases} 
\end{eqnarray*}
Notice that for any $1\le m <n$ the stratum $(\pi^n_m)^{-1}(\Lo_{m,\GG_m}(X))$ is isomorphic to the product $\AA^1_k\times_k f_n^{-1}(0)\cap (\pi^n_m)^{-1}(\Lo_{m,\GG_m}(X))$ by mapping an arc $\gamma$ in $X$ such that $f\circ \gamma=a_mt^m + \ldots + f_n(\gamma)t^n$ with $a_m\neq 0$ to the pair $(f_n(\gamma), \gamma\circ \vartheta_\gamma)$ with $\vartheta_\gamma$ being the automorphism of $k[t]/(t^{n+1})$ mapping $t$ to $t - \frac{f_n(\gamma)}{ma_m} t^{n-m+1}$. This isomorphism is actually $\GG_m$-equivariant if $g\in \GG_m$ acts on $(z,\gamma)\in \AA^1_k\times_k f_n^{-1}(0)\cap (\pi^n_m)^{-1}(\Lo_{m,\GG_m}(X))$ by means of $(g^nz,g\cdot \gamma)$. Hence, $[(\pi^n_m)^{-1}(\Lo_{m,\GG_m}(X))\longrightarrow X_0]$ is in the $\lambda$-ideal $\mathfrak{I}_n\subset \KK^{\GG_m,n}(\Var/\AA^1_{X_0})$ and $\int_{(\pi^n_m)^{-1}(\Lo_{m,\GG_m}(X))} \phi^{eq}_{f_n}=0$ follows. In particular, these strata will not contribute to $Z^{eq}_f(T)$, and we finally get
\begin{eqnarray*} \lefteqn{\LL^{d/2} Z^{eq}_f(\LL^{-d/2}T) \;=\; \sum_{n\ge 1} \Bigl( \bigl[ \Lo_n(X_0) \rightarrow X_0 \bigr] - \bigl[ \Lo_{n,1}(X) \rightarrow X_0 \bigr]\Bigr) \LL^{-nd} T^n } \\ 
& = &  J_{X_0}(\LL^{-d}T) - [X_0 \rightarrow X_0] - Z_f(T) \;=\; \frac{T[X_0\rightarrow X_0] - Z^{naive}_f(T)}{1-T} - Z_f(T),
\end{eqnarray*}
where we used equation (\arabic{equation}) and the Zeta function $Z_f(T)=\sum_{n\ge 1} \bigl[ \Lo_{n,1}(X)\rightarrow X_0\bigr] \LL^{-nd}T^n\in \hat{\KK}^{\muu}(\Var/X_0\times\mathbb{N}_{>0})$ introduced by Denef and Loeser \cite[Definition 3.2.1]{DenefLoeser2},\cite{DenefLoeser1}.  
Using motivic integration Denef and Loeser show that $Z^{naive}_f(T)$ and $Z_f(T)$ are Taylor series expansions around $T=0$ of rational functions in $\KK^{\muu}(\Var/X_0)[\LL^{-1/2}](T)$. Moreover, the rational functions are regular at $T=\infty$ and Denef and Loeser define the vanishing cycle sheaf (up to the factor $(-1)^d$) to be the value of $Z_f(T)$ at $T=\infty$ plus $[X_0 \xrightarrow{id} X_0]$ with trivial $\muu$-action. We will follow their definition but use a different factor. Notice that by the previous formula $Z^{eq}_f(T)$ is also a Taylor expansion around $T=0$ of a rational function without a pole at $T=\infty$.
\begin{Definition} If we also denote by $Z^{eq}_f(T)$ the rational function having Taylor series expansion around $T=0$ as above, the vanishing cycle sheaf is defined by 
\[ \phi_f=-Z^{eq}_f(\infty) \in \KK^{\muu}(\Var/X_0)[\LL^{-1/2}].\]
\end{Definition}
Our definition differs from that of Denef and Loeser by a factor $(-\LL^{1/2})^d$ which is a line element having Euler characteristic $1$. It coincides with that of Kontsevich and Soibelman \cite{KS1}. \\

To give an explicit expression for the rational function $Z^{eq}_f(T)$ we choose an embedded resolution of $X_0\subset X$, i.e.\ a smooth variety $Y$ together with a proper morphism $\pi:Y \rightarrow X$ such that $Y_0=(f\circ\pi)^{-1}(0)=\pi^{-1}(X_0)$ is a normal crossing divisor and $\pi: Y\setminus Y_0 \xrightarrow{\sim} X\setminus X_0$. Denote the irreducible components of $Y_0$ by $E_i$ with $i\in J$ and let $m_i>0$ be the multiplicity of $E_i$. Since $f\circ \pi$ is a section in $\mathcal{O}_Y(-\sum_{i\in J}m_iE_i)$, it induces a regular map to $\AA_k^1$ from the total space of $\mathcal{O}_Y(\sum_{i\in I} m_iE_i)$ for any $\emptyset \neq I \subset J$. The latter space restricted to $E_I^\circ$ is just $\otimes_{i\in I} N_{E_i|Y}^{\otimes m_i}|_{E_I^\circ}$. By composition with the tensor product we get a regular map $f_I:N_I:=\prod_{i\in I}( N_{E_i|Y}\setminus E_i) |_{E_I^\circ} \longrightarrow \AA_k^1$ which is obviously homogeneous of degree $m_i$ with respect to the $\GG_m$-action on the factor $(N_{E_i|Y}\setminus E_i)|_{E_I^\circ}$ and homogeneous of degree $m_I:=\sum_{i\in I}m_i$ with respect to the diagonal $\GG_m$-action. In particular, $f_I^{-1}(1)$ carries a natural $\muu$-action via $\mu_{m_I}$. By composing with $\pi:Y \rightarrow X$ the projection $N_I \rightarrow E_I^\circ$ along the fibers induces a $\muu$-equivariant map $f_I^{-1}(1) \rightarrow X_0$.\footnote{The diagonal action of $\GG_m$ on $N_I$ is canonical but not the only possible one. We can also let $\GG_m$ act with weight $w_i\in \mathbb{Z}$ on the fibers of $N_{E_i|Y}\setminus E_i$ for any $i\in I$ such that $w=\sum_{i\in I}w_im_i$ is positive. Then, $f_I$ is homogeneous of degree $w>0$ and $f_I^{-1}(1)$ carries a $\mu_w$-action. For different weights $(w_i)_{i\in I}$ we get different monodromy actions but one can show that the motives $[f_I^{-1}(1)\longrightarrow X_0]\in \KK^{\muu}(\Var/X_0)$ are always the same (see the proof of Theorem \ref{theorem1}).  A ``minimal'' monodromy action is obtained by choosing weights such that $\sum_{i\in I}w_im_i=\gcd(m_i\mid i\in I)$.}  Finally, we introduce positive integers $\nu_i$ by the formula $\pi^\ast K_X \otimes K_Y^{-1}=\mathcal{O}_Y(\sum_{i \in J}(1-\nu_i)E_i)$ with $K_X$ resp.\ $K_Y$ being the canonical divisors of $X$ resp.\ $Y$. By combining the explicit formulas for the rational functions associated to $Z^{naive}_f(T)$ and $Z_f(T)$ (see \cite[Theorem 3.3.1]{DenefLoeser2}, \cite[Corollary 3.3.2]{DenefLoeser2} and \cite[Lemma 5.3]{Looijenga1}) we immediately get the following theorem.

\begin{Theorem} \label{resolution}
Let $f:X\rightarrow \AA_k^1$ be as above and $\pi:Y \rightarrow X$ be an embedded resolution of $X_0$. In the notation just explained we have
\[ \LL^{d/2}Z^{eq}_f(\LL^{-d/2}T) = \sum_{\emptyset \neq \hat{I}\subset J\sqcup\{\star\} } a_{\hat{I}} \prod_{i\in \hat{I}} \frac{ \LL^{-\nu_i}T^{m_i}}{1-\LL^{-\nu_i}T^{m_i}} \]
with $\nu_\star=0$, $m_\star=1$, $a_{\{\star\}}=[X_0\rightarrow X_0]$, $a_{\hat{I}}=-[N_I\rightarrow X_0]$ for $\hat{I}=I\sqcup \{\star\}$ and $a_{\hat{I}}=-[N_I\rightarrow X_0] - [f_I^{-1}(1) \rightarrow X_0]$ for $\hat{I}=I\subset J$ with $\muu$ acting nontrivially only on $f_I^{-1}(1)$ via $\mu_{m_I}$. In particular,
\[ \phi_f \;=\; -\LL^{-\frac{d}{2}}\sum_{\emptyset \neq \hat{I}\subset J\sqcup\{\star\}} (-1)^{|\hat{I}|} a_{\hat{I}} \;= \; \LL^{-\frac{d}{2}}\Big( [X_0 \rightarrow X_0] + \sum_{\emptyset \neq I\subset J} (-1)^{|I|}[f_I^{-1}(1) \rightarrow X_0] \Big). \]
\end{Theorem}

The following proposition is a direct consequence of the previous theorem.
\begin{Proposition} \label{vanishing} \qquad
 \begin{enumerate}
  \item The motive $\phi_f$ is supported on $\Crit(f)\subset X_0$, in other words contained in $\KK^{\muu}(\Var/\Crit(f))[\LL^{-1/2}] \hookrightarrow \KK^{\muu}(\Var/X_0)[\LL^{-1/2}]$.\footnote{Notice: $\KK^{\muu}(\Var/\emptyset)=\{0\}$}
  \item If $f\equiv 0$, then $\phi_f=\LL^{-\frac{\dim X}{2}}[X \xrightarrow{id} X]$.
  \item Let $\pi:Y\rightarrow X$ be a smooth morphism of relative dimension $r$. Then $\phi_{f\circ\pi}=\LL^{-\frac{r}{2}}\pi^\ast \phi_f$.
 \end{enumerate}
\end{Proposition}
 
The last property is very useful for defining motivic vanishing cycles on Artin stacks. Indeed, let $\XX$ be a smooth Artin stack locally of finite type over $k$ and let $f:\XX \rightarrow \AA_k^1$ be a regular map and $\XX_0$ its zero locus. Assume that $\XX$ is locally a quotient stack, i.e.\ every closed point has an open neighborhood $\XX'$ of finite type which is isomorphic to a quotient stack $X/\Gl_k(n)$ for some smooth connected $k$-variety $X$ and some $n >0$.\footnote{By a result of Kresch \cite[Proposition 3.5.9]{Kresch} an Artin stack $\XX$ locally of finite type admits a locally finite stratification by quotient stacks if and only if it belongs to $\Sta/k$. Conjecturally, $\XX$ is locally of the form $X/\Gl(n)$ with smooth $X$ of finite type if and only if $\XX$ is smooth and  $\XX\in \Sta/k$.} If $\phi_f$ were defined such that Prop.\ \ref{vanishing} (3) holds for any representable smooth morphism between stacks of the form mentioned above, we would get for any local atlas $\pi:X \rightarrow X/\Gl_k(n)\cong \XX'\subset \XX$
\[ \pi_\ast \phi_{f\circ\pi}= \LL^{-\frac{n^2}{2}}\pi_\ast\pi^\ast \phi_f=\LL^{-\frac{n^2}{2}}\phi_f\times_\XX[X \xrightarrow{\pi} \XX]=\LL^{-\frac{n^2}{2}}[\Gl(n)] \,\phi_f|_{\XX'_0}\]
and, thus, $\phi_f|_{\XX'_0}=\LL^{\frac{n^2}{2}}[\Gl_k(n)]^{-1}\pi_\ast\phi_{f\circ\pi}$ in $\KK^{\muu}(\Var/\XX'_0)[[\Gl_k(n)]^{-1}, n \in \mathbb{N}]=\KK^{\muu}(\Sta/\XX'_0)$. We will take this as our (local) definition for $\phi_f$ and by Proposition \ref{vanishing} (3) this definition is independent of the choice of a local description by a quotient stack. In particular, the locally defined motives $\phi_f|_{\XX'_0}$ glue and we get a well-defined element $\phi_f\in \KK^{\muu}(\Sta/\XX_0)$. \\

As a special case we get motivic vanishing cycles for smooth Deligne--Mumford stacks of the form $X/G$ for some finite group $G$ acting on a smooth variety $X$ as $X/G\cong Y/\Gl_k(n)$ with $Y=X\times_G \Gl_k(n)$ for some embedding $G\hookrightarrow \Gl_k(n)$.  Let us firstly mention that any \'{e}tale locally trivial $G$-principal bundle on $\Spec (k[t]/(t^{p+1}))$ is trivial. Hence, $\Lo_p(X/G)=\Lo_p(X)/G$.
By left translation on $\Gl_k(n)$ we have $\Lo_p(X\times_k \Gl_k(n))=\Lo_p(X)\times_k \Gl_k(n) \times_k \gl_k(n)^p$ with $G$ acting nontrivially only on the factor $\Lo_p(X)\times_k \Gl_k(n)$. As $X\times_k \Gl_k(n) \longrightarrow Y$ is an \'{e}tale $G$-principal bundle, $\Lo_p(Y)=\Lo_p(X\times_k \Gl_k(n))/G=\Lo_p(X)\times_G \Gl_k(n) \times_k \gl_k(n)^p$. Thus, $\Lo_p(Y) \longrightarrow \Lo_p(X)/G$ is a vector bundle of rank $n^2p$ over a $\Gl_k(n)$-principal bundle on $\Lo_p(X)/G$. By Proposition \ref{basicprop} (5) we finally get for any regular $f:X/G \longrightarrow \AA^1_k$ with associated $\tilde{f}:Y\longrightarrow \AA^1_k$ and $f_p:\Lo_p(X)/G \longrightarrow \AA^1_k$
\[ Z^{eq}_{\tilde{f}}(T)=\frac{[\Gl_k(n)]}{\LL^{n^2/2}} \sum_{p\ge 1} \int_{\Lo_p(X)/G|_{X_0/G}} \phi^{eq}_{f_p} (\LL^{n^2/2}T)^p =: \frac{[\Gl_k(n)]}{\LL^{n^2/2}} \, Z^{eq}_f(\LL^{n^2/2}T).\]
In particular, the ``stacky'' generating series $Z^{eq}_f(T)$ is a Taylor series expansion of a rational function with regular value $-\phi_f$ at $T=\infty$. But the reader should be warned. The rational function is not obtained by applying Theorem \ref{resolution} to a $G$-equivariant resolution of $X$ and by replacing any space with $G$-action with its quotient stack.\\

We will apply this to the following situation. Start with a regular function $f:X\rightarrow \AA^1_k$ on a smooth variety of dimension $d$. Form the smooth Deligne--Mumford stack $\Symm^nX =  X^n/S_n$ and the regular function $\SSymm_+^n(f): \Symm^n X \xrightarrow{\Symm^n(f)} \Symm^n \AA^1 \xrightarrow{\;+\;} \AA^1_k$. Notice that $\Symm^n X_0$ is naturally a substack of $(\Symm^n X)_0$.

\begin{Theorem}  \label{symvanishing}
If we consider $\phi_f\in \KK^{\muu}(\Var/X_0)[\LL^{-1/2}]$ and $\phi_{\SSymm_+^n(f)}|_{\Symm^n X_0}\in \KK^{\muu}(\Sta/\Symm^n X_0)$ as elements of the $\lambda$-ring $\hat{\KK}^{\muu}(\Sta/\Symm X_0)$ via the obvious maps, we get for any $n\ge 0$ the equation
\[ \LL^{nd/2} \phi_{\SSymm_+^n(f)}|_{\Symm^n X_0} = \sigma^n ( \LL^{d/2} \phi_f). \]
\end{Theorem}

\begin{proof}
Notice that $\Lo_p(\Symm^n X)= \Lo_p(X^n/S_n)=\Lo_p(X)^n/S_n =\Symm^n \Lo_p(X)$, $(\pi^p)^{-1}(\Symm^n X_0)=\Symm^n ((\pi^p)^{-1}(X_0))$ and $(\SSymm_+^n(f))_p=\SSymm_+^n(f_p)$. By Proposition \ref{basicprop} (4) with $\MM:=\Symm X_0$ we get
\begin{eqnarray*} \lefteqn{\LL^{nd/2} Z^{eq}_{\SSymm_+^n(f)}(\LL^{-nd/2}T)|_{\Symm^n X_0} } \\ 
& = & \sum_{p\ge 1} \LL^{-ndp}\cdot \LL^{\dim \Symm^n\Lo_p(X)/2} \int_{\Symm^n ((\pi^p)^{-1}(X_0))} \phi^{eq}_{\SSymm_+^n(f_p)} T^p \\
& = & \sum_{p\ge 1} \sigma^n\Bigl( \LL^{-dp} \cdot\LL^{\dim \Lo_p(X)/2} \int_{(\pi^p)^{-1}(X_0)} \phi^{eq}_{f_p} \Bigr) T^p \\
& = & \sigma^n\bigl( \LL^{d/2} Z^{eq}_f (\LL^{-d/2}T) \bigr),
\end{eqnarray*}
where we used the $\lambda$-ring structure on $\hat{\KK}^{\muu}(\Sta/\Symm X_0\times \mathbb{N}_{>0})$ induced by the obvious structure of $\Symm X_0\times \mathbb{N}_{>0}$ as a monoid in the category of stacks over $\mathbb{N}_{>0}$ (cf.\ Remark \ref{relative}). If we identify this $\lambda$-ring with $TR[[T]]$ for $R:=\hat{\KK}^{\muu}(\Sta/\Symm X_0)$, the multiplication has to be done coefficient wise and, similarly, $\sigma^n$ has to be applied to every single coefficient (without changing $T$) for $n\ge 0$. We can also form the $\lambda$-ring $R[[T,T^{-1}]]\cong \hat{\KK}^{\muu}(\Sta/\Symm X_0 \times \mathbb{Z})$ with obvious $\lambda$-ring homomorphisms to $TR[[T]]$ resp.\ to $R$ given by restriction to $\mathbb{N}_{>0}$ resp.\  to $\{0\}$. Denote by $R\langle T \rangle$ the $R$-module of rational functions 
\[ g(T)=\sum_{\Lambda} g_{\Lambda}\prod_{i\in \Lambda}\frac{\LL^{-a_i}T^{b_i}}{1- \LL^{-a_i}T^{b_i}}, \] where the sum is over a finite collection of index sets $\Lambda$, $g_\Lambda \in R$ and $a_i,b_i\in \mathbb{N}$ with $b_i>0$ for all $i\in \Lambda$. Denote by $g_0(T)\in TR[[T]]$ resp. $g_\infty(T)\in R[[T^{-1}]]$ the Taylor series expansions of $g(T)\in R\langle T \rangle$ in $T=0$ resp.\ $T=\infty$ and define $\tau:R\langle T \rangle \ni g(T) \longmapsto g_0(T)-g_\infty(T) \in R[[T,T^{-1}]]$. Obviously, $\tau$ and its composition $\tau_0:R\langle T\rangle \ni g(T) \longmapsto g_0(T)\in TR[[T]]$ with $R[[T,T^{-1}]]\longrightarrow TR[[T]]$ are injective. Let $R[[T,T^{-1}]]_{rat}$ and $TR[[T]]_{rat}$ denote the images of $\tau$ and $\tau_0$. Thus, $R\langle T\rangle \cong R[[T,T^{-1}]]_{rat}\cong TR[[T]]_{rat}$. \\   
By Theorem \ref{resolution}, $\LL^{d/2}Z^{eq}_f(\LL^{-d/2}T)=g_0(T)$ for some $g(T)\in R\langle T \rangle$. The theorem is now a direct consequence of the following lemma, the definition of the vanishing cycle sheaf and the calculations above.
\end{proof}

\begin{Lemma} \label{conjecture}
The $R$-submodules $TR[[T]]_{rat}$ and $R[[T,T^{- 1}]]_{rat}$ of $TR[[T]]$ resp.\ $R[[T,T^{-1}]]$ are $\lambda$-subrings. In particular, the map $TR[[T]]_{rat}\ni g_0(T) \longmapsto -g_\infty(0)=-g(\infty) \in R$ is a $\lambda$-ring homomorphism.
\end{Lemma} 
\begin{proof}
One can use the Euclidean algorithm and the fact that $1-\LL^a$ is invertible in $R=\hat{\KK}^{\muu}(\Sta/\Symm X_0)$ for every $0\neq a\in \mathbb{Z}$ to expand any product $g_\Lambda \prod_{i\in \Lambda} \frac{\LL^{-a_i}T^{b_i}}{1-\LL^{-a_i}T^{b_i}}$ into partial fractions $\sum_{j\in J} g_j\frac{T^{r_j}\LL^{-a_j}T^{b_j}}{(1-\LL^{-a_j}T^{b_j})^{l_j}}$ with $g_j\in R$, $a_j,b_j,l_j,$ $r_j\in \mathbb{Z}$ such that $b_j,l_j>0$, $r_j\ge 0$. Thus, $R[[T,T^{-1}]]_{rat}$ is the $R$-linear span of elements of the form
\[ \sum_{m\in \mathbb{Z}} { m+l-1 \choose l-1 } \LL^{-am}T^{bm+r} \;\mbox{ or more generally }\; \sum_{m\in \mathbb{Z}} f(m) \LL^{-am}T^{bm+r} \]
with $a,b,r,l \in\mathbb{Z}$ such that $b,l>0$, $r\ge 0$ and $f\in \mathbb{Q}[X]$ with $f(\mathbb{Z})\subset \mathbb{Z}$ because the space of these functions is spanned over $\mathbb{Z}$ by binomial coefficients. Notice that the space of these rational polynomials is closed under multiplication and composition. If the product of  two series as above is not zero, we get 
\[ \sum_{m\in \mathbb{Z}}f(m)\LL^{-am}T^{bm+r} \cdot \sum_{n\in\mathbb{Z}} f'(n)\LL^{-a'n}T^{b'n+r'} = \sum_{p\in \mathbb{Z}} f''(p)\LL^{-a''p}T^{b''p+r''}\]
with $b''=\lcm(b,b')$, $a''=ab''/b + a'b''/b'$, $r''=bm_0+r=b'n_0+r'$ and $f''(p)=f(pb''/b+m_0)f'(pb''/b'+n_0)$ where $(m_0,n_0)\in \mathbb{N}_{>0}^2$ is the smallest solution of $bm+r=b'n+r'$. This part of the lemma has already been proven by Denef and Loeser (see Prop.\ 5.1.1. and Prop.\ 5.1.2 in \cite{DenefLoeser3}). \\
Using the basic properties of the $\lambda$-ring $R$ we also obtain
\begin{eqnarray*} \lefteqn{ \sigma^n\Bigr( g \sum_{m\in\mathbb{Z}} f(m) \LL^{-am}T^{bm+r}\Bigl) \;=\; \sum_{m\in \mathbb{Z}} \sigma^n(gf(m))\LL^{-anm}T^{bm+r} } \\ 
& & = \sum_{m\in \mathbb{Z}} P^n(\sigma^1(g),\ldots, \sigma^n(g),\sigma^1(f)(m),\ldots, \sigma^n(f)(m)) \LL^{-anm}T^{bm+r}
\end{eqnarray*}
for any $g\in R$, where we used the universal polynomials $P^n$ from the definition of a special $\lambda$-ring as well as the polynomials $\sigma^n(f)\in \mathbb{Q}[X]$ with values $\sigma^n(f)(m)={ f(m) + n-1 \choose n}$ in $\mathbb{Z}$ for every $m\in \mathbb{Z}$ (cf.\ Example \ref{lambdarings} (1)). As $R[[T,T^{-1}]]_{rat}$ is closed under multiplication, it is, therefore, also a $\lambda$-subring. The same arguments show that $TR[[T]]_{rat}$ is a $\lambda$-subring of $TR[[T]]$.
\end{proof}

The following theorem generalizes the classical Thom--Sebastiani theorem to the motivic vanishing cycle.

\begin{Theorem}[\cite{DenefLoeser3}, Theorem 5.2.2] \label{ThomSebastiani}
Let $X$, $Y$ be smooth varieties and let $f:X\rightarrow \AA_k^1$ and $g:Y \rightarrow \AA_k^1$ be regular maps. Define the map $f+g:X\times_k Y\ni (x,y) \longmapsto f(x)+g(y)\in  \AA_k^1$ and observe $X_0\times_k Y_0 \subset (X\times_k Y)_0$. Then  
\[\phi_{f+g}|_{X_0\times_k Y_0}= \phi_f\boxtimes_k\phi_g, \] 
where we used the obviously defined exterior product of motives.
\end{Theorem}
\begin{proof}
Using Proposition \ref{basicprop} (3) the proof follows exactly the same lines as in the proof of Theorem \ref{symvanishing}. \\
One could also deduce the theorem directly from Theorem \ref{symvanishing} by replacing $f:X\rightarrow \AA_k^1$ in Theorem \ref{symvanishing} with the unique morphism $h: X\sqcup Y \rightarrow \AA_k^1$ which is $f$ on $X$ and $g$ on $Y$. As $X_0\times_k Y_0 \hookrightarrow \Symm^2(X\times_k Y)_0$ Theorem \ref{symvanishing} will prove the Thom--Sebastiani theorem because $\SSymm_+^2(h)$ coincides with $f+g$ on $X\times_kY \hookrightarrow \Symm^2(X\sqcup Y)$. 
\end{proof} 
\begin{Remark}
It is not difficult to see that the critical locus of $\SSymm_+^n(f)$ coincides with $\Symm^n\Crit(f)\subset \Symm^n X_0 \subset (\Symm^n X)_0$ and similarly $\Crit(f+g)=\Crit(f)\times_k \Crit(g)$. By Proposition \ref{vanishing} (1) we can, therefore, ignore the restriction to $\Symm^n X_0$ in Theorem \ref{symvanishing} resp.\ to $X_0\times_k Y_0$ in Theorem \ref{ThomSebastiani}. 
\end{Remark}

For a locally closed subvariety $X^s\subset X$ we define the element
\[  \int_{X^s} \phi_f:= \int_{f^{-1}(0)\cap X^s} \phi_f|_{f^{-1}(0)\cap X^s} \]
in $\KK^{\muu}(\Var/k)[\LL^{-1/2}]$. As an analogue to Proposition \ref{basicprop} we have the following proposition. 
\begin{Proposition} \label{basicprop2}
Let $X,Y$ and $Z$ be equidimensional smooth varieties of finite type over $\BB$. Assume, moreover, that $X^s,Y^s,Z^s$ are locally closed subvarieties in $X,Y$ resp.\ $Z$. Let $\pi:Z \rightarrow X$, $f:X \rightarrow \AA^1_\BB$ and $g:Y \rightarrow \AA^1_\BB$ be regular maps and assume that $\pi$ is smooth. Then
\begin{enumerate}
\item $\int_{X^s}\phi_f=\LL^{-\frac{\dim X}{2}} [X^s]$ if $f\equiv 0$,
\item $\int_{X^s\sqcup X^s_2} \phi_f=\int_{X^s}\phi_f + \int_{X^s_2} \phi_f$ for any locally closed subvariety $X^s_2\subset X$ not intersecting $X^s$,
\item $\int_{X^s\times Y^s}\phi_{f+g}=\int_{X^s}\phi_f \cdot \int_{Y^s}\phi_g$ (Thom--Sebastiani),
\item $\LL^{n\dim X/2}\int_{\Symm^n X^s} \phi_{\SSymm_+^n(f)}=\sigma^n( \LL^{\dim  X/2}\int_{X^s}\phi_f)$ for all $n \ge 0$,
\item $\int_{Z^s}\phi_{f\circ \pi}= \LL^{-\frac{r}{2}}[\mathcal{F}]\int_{X^s}\phi_f$ if $\pi|_{Z^s}:Z^s\rightarrow X^s$ is a Zariski locally trivial fibration with $r$-dimensional fiber $\mathcal{F}$. The same holds if $\pi|_{Z^s}$ is a vector bundle or a $\Gl_k(n)$-principal bundle.
\end{enumerate}
\end{Proposition}
The properties (1) to (4) follow directly from the definition and properties stated before. To show (5) we mention $\phi_{f\circ \pi}|_{Z^s}=\LL^{-r/2}(\pi^\ast \phi_f)|_{Z^s}=\LL^{-r/2}(\pi|_{Z^s})^\ast \phi_f$ by Proposition \ref{vanishing} (3). Now apply the projection formula. \\

The following theorem is very useful for computing the motivic vanishing cycle in a $\GG_m$-equivariant situation. Let $X$ be a smooth equidimensional variety locally of finite type equipped with a $\GG_m$-action. Let us assume that every closed point has an open neighborhood which as a variety with $\GG_m$-action is isomorphic to $\AA_k^r\times Z$ with $\GG_m$ acting via $g\cdot(v_1,\ldots,v_r,z)=(g^{w_1}v_1,\ldots, g^{w_r}v_r, z)$ for $g\in G, (v_1,\ldots,v_r)\in \AA^r_k, z\in Z$ with strictly positive weights $w_1,\ldots,w_r$. In particular, $Z$ is the intersection of the neighborhood with the fixed point set $X^{\GG_m}$ and, hence, smooth. Moreover, it is not difficult to see that the projection to $Z$ along $\AA_k^r$ can be described by $\lim_{g \to 0} g\cdot x$ and extends, therefore, to a smooth map $X \rightarrow X^{\GG_m}$. The assumption is not very restrictive, as any smooth projective variety with $\GG_m$-action has a dense open subset satisfying our assumption by a theorem of Bia{\l}ynicki--Birula \cite{Bialynicki-Birula1},\cite{Bialynicki-Birula2}. Conjecturally, any smooth quasiprojective variety $X$ with $\GG_m$-action such that $X^{\GG_m}$ is connected and $\lim_{g \to 0} g\cdot x$ exists for any closed point $x$ should satisfy our assumption. 
\begin{Theorem} \label{theorem1}
 Let $X$ be a smooth variety with $\GG_m$-action satisfying the assumption mentioned above with all weights equal to one. Let $f:X \rightarrow \AA_k^1$ be a $\GG_m$-equivariant morphism of degree $d$, i.e.\ $f(g\cdot x)=g^df(x) \;\forall g\in\GG_m, x\in X$. Let $\muu$ act on $f^{-1}(1)$ via $\mu_d$ and trivially on $f^{-1}(0)$. Then
 \[ \int_{X} \phi_f = \int_{X}\phi^{eq}_f=\LL^{-\frac{\dim X}{2}}\bigl([f^{-1}(0)] - [f^{-1}(1)] \bigr) \;\;\mbox{in} \;\KK^{\muu}(\Var/k)[\LL^{-1/2}].\]
\end{Theorem}
\begin{proof}
Let us first assume $X=\AA_k^r\times Z$ with $\GG_m$ acting nontrivially only on the affine ``fiber'' $\AA_k^r$ by scalar multiplication. Consider the blow-up $\tilde{X}$ of $Z$ in $X$ which has a natural fibration towards the exceptional divisor $\tilde{E}_0$ induced by the affine fibration $X \rightarrow Z$. Moreover, the $\GG_m$-action has a lift to $\tilde{X}$ with fixed point set $\tilde{E}_0$ and it is not difficult to see that $\tilde{X}$ as a variety with $\GG_m$-action is isomorphic to the normal bundle $N:=N_{\tilde{E}_0|\tilde{X}}$ with $\GG_m$ acting by scalar multiplication on the fibers. Denote by $\tilde{E_i}$ for $i=1,\ldots, l$ the strict transforms of the irreducible components of the divisor $f^{-1}(0)$. As they are closed and $\GG_m$-invariant, we get $\tilde{E_i}=N|_{\tilde{D}_i}$ for the divisors $\tilde{D}_i=\tilde{E}_0\cap\tilde{E}_i$ in $\tilde{E}_0$. Note that the collection $(\tilde{D}_i)_{i=1}^l$ might not be a normal crossing divisor in $\tilde{E}_0$. \\
Let $\sigma:E_0 \rightarrow \tilde{E}_0$ be an embedded resolution of the $\tilde{D}_i$, i.e.\ the strict transforms $D_i \;(i=1,\ldots,l)$ together with the exceptional divisors $D_{l+1},\ldots, D_m$ of $\sigma$ form a normal crossing divisor in $E_0$. Moreover, $E_0$ is smooth and $\sigma: E_0\setminus \cup_{i=1}^m D_i \xrightarrow{\;\sim\;} \tilde{E}_0\setminus \cup_{i=1}^l\tilde{D}_i$. Consider the pull-back $Y:=\sigma^\ast N$ and the normal crossing divisors $E_0, \:E_i:=(\sigma^\ast N)|_{D_i}$ for $i=1, \ldots, m$ in $Y$ along with the proper morphism given by the composition $\pi: Y=\sigma^\ast N \xrightarrow{\sigma} N=\tilde{X} \rightarrow X$. By construction $(f\circ\pi)^{-1}(0)=\cup_{i=0}^m E_i$ set-theoretically and $\pi: Y\setminus \cup_{i=0}^m E_i \xrightarrow{\;\sim\;} X\setminus f^{-1}(0)$. We will use this embedded resolution of $X_0=f^{-1}(0)$ to compute $\int_{X} \phi_f$. \\
As $Y\xrightarrow{p} E_0$ is a line bundle, we get $N_{E_0|Y}=Y$ and $N_{\{0\}}=Y\setminus \cup_{i=0}^m E_i$. The induced map $f_{\{0\}}$ is just $f\circ \pi$ as the latter is homogeneous and $E_0$ is of multiplicity $d$. Moreover, after identifying $N_{\{0\}}$ with $X\setminus f^{-1}(0)$ by means of $\pi$, we get $f_{\{0\}}=f$ on $N_{\{0\}}$ and, thus, $f_{\{0\}}^{-1}(1)=f^{-1}(1)$ with $\muu$-action given by the natural $\mu_d$-action on $f^{-1}(1)$. \\
On the other hand, for any $i=1, \ldots, m$ we have by construction
\[ N_{E_i|Y}=p^\ast N_{D_i|E_0}= N_{E_i|Y}|_{E_0\cap E_i} \times N_{E_0|Y}|_{E_0\cap E_i} \] and, thus, for any $\emptyset\neq I\subset \{1,\ldots,m\}$ 
\[ N_I=\prod_{i\in I} (N_{E_i|Y}\setminus E_i)|_{E_I^\circ} =  \prod_{i\in I\cup \{0\}} ( N_{E_i|Y}\setminus E_i)|_{E^\circ_{I\cup\{0\}}}=N_{I\cup\{0\}} .\]
Moreover, by $\GG_m$-equivariance of $f\circ \pi$ the induced maps $f_I$ and $f_{I\cup \{0\}}$ coincide, which can be checked by a local calculation. Hence, $f_I^{-1}(1)=f_{I\cup\{0\}}^{-1}(1)$. Unfortunately, the $\muu$-actions are different but, nevertheless, the $\muu$-equivariant motives $[f_I^{-1}(1)]$ and $[f_{I\cup\{0\}}^{-1}(1)]$ coincide. To see this we choose local functions $(z_i)_{i\in I\cup\{0\}}$ in some $\GG_m$-invariant neighborhood $V=p^{-1}(V\cap E_0)$ of $y\in E^\circ_{I\cup\{0\}}$ such that $E^\circ_{I\cup\{0\}}\cap V$ is the zero locus of $z_0\prod_{i\in I} z_i$. Hence, $N_{I}|_V\cong N_{I\cup\{0\}}|_V \cong  E^\circ_{I\cup\{0\}}\cap V \times \GG_m \times \GG_m^I$ and $f_I=f_{I\cup\{0\}}=uz_0^d\prod_{i\in I}z_i^{m_i}$ with $u$ being a unit on $V$, $m_i>0$ being the multiplicities of $E_i$ in $(f\circ \pi)^{-1}(0)$ and with $z_i$ identified with the coordinates on the corresponding ``normal'' $\GG_m$-factors. The action of $\muu$ via $\mu_{m_I}$ with $m_I=\sum_{i\in I}m_i$ is given by diagonal embedding of $\mu_{m_I}$ into $\GG_m^I$ and similarly for $\mu_{m_{I\cup\{0\}}}=\mu_{m_I+d}$. \\
However, we can choose an automorphism of $\GG_m\times \GG_m^I$ mapping $(z_i)_{i\in I\cup\{0\}}$ to $(\prod_{j\in I\cup\{0\}} z_j^{a_j^{(i)}})_{i\in I\cup\{0\}}$ with $(a^{(i)})_{I\cup\{0\}}$ being a basis of the group $\mathbb{Z}\times \mathbb{Z}^{I}$ of characters of $\GG_m\times \GG_m^I$ such that $a^{(0)}_i=m_i/e$ for all $i\in I\cup\{0\}$ with $e:=\gcd(m_j\mid j\in I\cup\{0\})$ and $m_0=d$. After this coordinate change on $N_{I}|_V$ the function $f_I$ is given by $uz_0^e$. Using relation (\ref{relation}) we see that $[f_I^{-1}(1)]$ is given by   $(\LL-1)^{|I|}[\tilde{E}^\circ_{I\cup\{0\}}]$ with $\tilde{E}^\circ_{I\cup\{0\}}$ being a Galois cover of $E^\circ_{I\cup\{0\}}$ with Galois group $\mu_e$, locally given by $\{(z,y)\in \AA^1_k\times E^\circ_{I\cup\{0\}}\cap V \mid z^eu(y)=1 \}$. Moreover, the group $\mu_{m_I}$  acts by its quotient group $\mu_e$. Exactly the same holds for $f_{I\cup\{0\}}^{-1}(1)$ acted on by $\mu_{m_{I\cup \{0\}}}$. \\
Thus, $[f_I^{-1}(1)]=[f_{I\cup\{0\}}^{-1}(1)]$ in $\KK^{\muu}(\Var/k)$ and their contributions to $\int_{X} \phi_f$ cancel.\footnote{Note that the projections to $X_0$ are different so that they do not cancel each other in the formula for $\phi_f$.} In the formula for $\int_{X} \phi_f$ we are left with the contribution $f_{\{0\}}^{-1}(1)=f^{-1}(1)$ which proves the theorem.\\

The general case is a purely combinatorial argument using the motivic behavior of the integral. Indeed, for any $\GG_m$-invariant open subset $U=\AA_k^r\times Z$  in $X$ we get by the previous arguments and Proposition \ref{basicprop} (5) and Proposition \ref{basicprop2} (5) 
\begin{equation} \int_{U}\phi_f  = \int_{U} \phi_{f|_{U}} = \int_{U}\phi^{eq}_{f|_{U}} = \int_{U}\phi^{eq}_f .\end{equation}
Let us now take a general smooth variety $X$ with $\GG_m$-action satisfying the assumptions of the theorem. Choose an open covering by $\GG_m$-invariant subsets of the form $U_i=\AA_k^{r_i}\times Z_i$ $(i\in I)$. If $U_i \cap U_j \neq \emptyset$, the intersection is of the form $\AA_k^{r_i} \times \tilde{Z}_i$ for some open $\tilde{Z}_i\subset Z_i$ as $U_i \cap U_j$ must contain the limits $\lim_{g \to 0} g\cdot x$ for all $x\in U_i\cap U_j$. By applying equation (\arabic{equation}) to any non-empty intersection $U_J=\cap_{i\in J}U_i$ we finally get  
\[ \int_{X}\phi_f = \sum_{U_J\neq \emptyset} c_J \int_{U_J}\phi_f =  \sum_{U_J\neq \emptyset} c_J \int_{U_J}\phi^{eq}_f = \int_{X}\phi^{eq}_f \]
using Proposition \ref{basicprop} (2) and Proposition \ref{basicprop2} (2), where the $c_J\in \mathbb{Z}$ are coefficients depending only on the combinatorics of the subsets $J\subset I$. 
\end{proof}
Unfortunately, we were not able to prove the theorem for arbitrary positive weights $w_1,\ldots,w_r> 0$. But we strongly believe that Theorem \ref{theorem1} holds also in that case.

\section{The one loop quiver with potential}

In this section we apply the results obtained in the previous parts to compute the motivic Donaldson--Thomas invariants of the one loop quiver with potential $W\in k[t]$. To simplify the notation we will assume that $\KK^{\muu}(\ldots/\mathbb{N}^r)$ for $r\ge 0$ is already completed with respect to the topology introduced in Remark \ref{Kaiso}. Let
\[ \XX = \coprod_{n\ge 0} \Mat_k(n,n)/\Gl_k(n) \]
be the stack of finite dimensional representations of the one loop quiver, i.e. representations of the ring $k[t]$. Equivalently, it can be seen as the stack parametrizing zero dimensional sheaves on $\AA^1_k$. To a given potential $W\in k[t]$ we associate for any $n\ge 0$ the $\Gl_k(n)$-equivariant function $W_n:\Mat_k(n,n)\ni A \longmapsto \tr W(A) \in \AA^1_k$. They induce a regular function $\WW:\XX \rightarrow \AA^1_k$ and its critical locus $\MM:=\Crit(\WW)$ is the stack we are interested in. The stacks $\XX$ and $\MM$ are natural monoids over $k$ as for any two representations we can form their direct sum. Let us also consider  the monoid homomorphism $\dim:\XX \rightarrow \mathbb{N}$ mapping each representation to its dimension or, equivalently, each sheaf to its length. Notice that the substack $\Crit(\WW)$ parametrizes sheaves supported (scheme-theoretically) on the zero-scheme $\Spec (k[t]/(W')) \subset \AA^1_k$ of $W'$. If we denote the line element $[\Spec (k) \rightarrow 1]\in \KK^{\muu}(\Sta/\mathbb{N})$ by $T$, we obtain
\[ \Phi_W(T):= \int_{\dim} \phi_\WW = \sum_{n\ge 0} \frac{\int_{\Mat_k(n,n)} \phi_{W_n}}{\LL^{-n^2/2}[\Gl_k(n)]}\,T^n = \Sym \Bigl( \frac{1}{\LL^{1/2} - \LL^{-1/2}} \sum_{n\ge 1} \Omega_n T^n \Bigr)  \]
in $\KK^{\muu}(\Sta/\mathbb{N})$ by definition of the Donaldson--Thomas invariants $\Omega_n$ (see \cite{Meinhardt3}). The stack $\XX$ carries a natural good $\GG_m$-action given by scalar multiplication on the space of matrices. If $W$ is homogeneous, we can, therefore, also consider the element\footnote{In the equivariant case we consider $\XX$ as a stack over $\mathbb{N}$.}
\[ \Phi_W^{eq}(T):=\int_{\XX} \phi^{eq}_{\WW}=\sum_{n\ge 0} \frac{\int_{Mat_k(n,n)}\phi^{eq}_{W_n}}{\LL^{-n^2/2}[\Gl_k(n)]}\,T^n  \]
in $\KK^{\muu}(\Sta/\mathbb{N})$, where we applied Proposition \ref{basicprop} (5). To compute the series $\Phi_W(T)$ and $\Phi_W^{eq}(T)$ we use the stack 
\[ \Hilb(\AA^1_k) = \coprod_{n\ge 0} H_n/\Gl_k(n) \cong \coprod_{n\ge 0 } \AA^n_k \]
with $H_n:=\{ (A,v)\in \Mat_k(n,n)\times \AA_k^n \mid \Span_k(v,Av,\ldots,A^{n-1}v)=k^n\}$ equipped with the obvious $\Gl_k(n)$-action. Let $\GG_m$ act on $H_n$ by scalar multiplication and on $\AA^n_k$ with weights $1,\ldots,n$. The isomorphism is induced by the $\Gl_k(n)$-principal bundle $q_n:H_n\ni (A,v) \longmapsto (\tr A^1, \ldots, \tr A^n)\in \AA^n_k$.  If for any $A\in \Mat_k(n,n)$ we express $\tr W(A)$  in terms of $\tr(A^i)$ with $i\in \{1,\ldots,n\}$, we obtain functions $f_n:\AA_k^n \rightarrow \AA_k^1$ making the diagram     
\[ \xymatrix { H_n \ar[r]^{q_n} \ar[d]_{p_{\Mat}} & \AA_k^n \ar[d]^{f_n} \\ \Mat_k(n,n) \ar[r]^-{W_n} &  \AA^1_k } \]
commutative and $\GG_m$-equivariant if $W$ is homogeneous. 
\begin{Lemma} \label{wallcrossing}
 Using the notation just introduced, one has
\begin{eqnarray*} & &  \Phi_W(\LL T)/ \Phi_W(T) =\sum_{n\ge 0} \int_{\AA^n_k} \phi_{f_n} (\LL^{1/2}T)^n \mbox{ and }\\
 & &  \Phi_W^{eq}(\LL T)/ \Phi_W^{eq}(T) =\sum_{n\ge 0} \int_{\AA^n_k}\phi^{eq}_{f_n} (\LL^{1/2}T)^n \mbox{ for homogeneous } W
\end{eqnarray*}
in  $\KK^{\muu}(\Sta/\mathbb{N})\cong \KK^{\muu}(\Var/k)[[\Gl_k(n)]^{-1}, n\in \mathbb{N}][[T]].$
\end{Lemma} 
\begin{proof}
The key observation is the following formula in the (equivariant) Hall algebra $\KK^{(\GG_m)}(\Sta/\XX)$ first observed by Reineke (cf.\ \cite[Lemma 5.1]{Reineke1})
\[ \left[\frac{\Mat_k(\Nn,\Nn)\times \AA^\Nn_k}{\Gl_k(\Nn)} \xrightarrow{p_{\Mat}}  \XX\right]=\sum_{n=0}^\Nn \left[\frac{H_n}{\Gl_n} \xrightarrow{p_{\Mat}} \XX\right] \star \left[\frac{\Mat_k(\Nn-n,\Nn-n)}{\Gl_k(\Nn-n)} \hookrightarrow \XX\right]. \]
We apply the algebra homomorphisms 
\begin{eqnarray*}
 \int_{\dim}^{\phi_{\WW}}: \KK(\Sta/\XX) \ni [\ZZ \xrightarrow{\pi} \XX] & \longmapsto & \int_{\dim\circ\pi} \pi^\ast \phi_{\WW} \in  \KK^{\muu}(\Sta/\mathbb{\mathbb{N}})\\
 \int_{\dim}^{\phi^{eq}_{\WW}} : \KK^{\GG_m}(\Sta/\XX) \ni [\ZZ \xrightarrow{\pi} \XX] & \longmapsto & \int_{\ZZ  } \phi^{eq}_{\WW\circ \pi} \in \KK^{\muu}(\Sta/\mathbb{\mathbb{N}})
\end{eqnarray*}
from the (equivariant) Hall algebras to $\KK^{\muu}(\Sta/\mathbb{\mathbb{N}})$ (see \cite{Meinhardt3}). By Proposition \ref{basicprop} (5), Proposition \ref{vanishing} (3) and Proposition \ref{basicprop2} (5) 
\begin{eqnarray*} \int_{\dim}^{\phi_{\WW}}\left[\frac{\Mat_k(\Nn,\Nn)\times \AA^\Nn_k}{\Gl_k(\Nn)} \xrightarrow{p_{\Mat}} \XX\right] &=& \LL^\Nn \int_{\dim}^{\phi_{\WW}} \left[\frac{\Mat_k(\Nn,\Nn)}{\Gl_k(\Nn)} \hookrightarrow \XX\right] \mbox{ and} \\
 \int_{\dim}^{\phi_{\WW}} \Bigl[H_n/\Gl_k(n) \xrightarrow{p_{\Mat}} \XX\Bigr] &=& \LL^{n/2}\int_{\AA^n_k} \phi_{f_n} 
\end{eqnarray*}
and similarly for the equivariant version. Multiplying with $T^{\Nn}$ and summing over $\Nn\ge 0$ proves the lemma. Notice that the Lemma can also be seen as a special wall-crossing formula (cf. \cite{Mozgovoy2}).  
\end{proof}

To compute the integral $\LL^{n/2}\int_{\AA^n_k} \phi_{f_n}$ we restrict ourselves firstly to the case $W=t^d$ of normalized homogeneous potentials. By Theorem \ref{theorem1} $\Phi_W(T)=\Phi^{eq}_W(T)$ and  
\[ \LL^{n/2}\int_{\AA^n_k} \phi_{f_n}=\LL^{n/2}\int_{\AA^n_k}\phi^{eq}_{f_n} \] 
follows from the previous lemma. Notice that $\Sym^n \AA_k^1 \cong \AA^n_k$, induced by 
\[ \tilde{q}_n:\AA^n_k\ni(z_1,\ldots,z_n)\longmapsto (z_1+\ldots+z_n, \ldots, z_1^n+\ldots+ z_n^n)\in \AA_k^n. \]
An easy calculation shows $f_n\circ \tilde{q}_n(z_1, \ldots, z_n)=z_1^d + \ldots + z_n^d$, in other words, $f_n=\SSym_+^n(W)$. By Proposition \ref{basicprop} (4) we finally obtain
\[ \LL^{n/2}\int_{\AA^n_k} \phi_{f_n}=\LL^{n/2}\int_{\AA^n_k} \phi^{eq}_{f_n} = \sigma^n\Bigl(\LL^{1/2}\int_{\AA^1_k}\phi^{eq}_{W}\Bigr)=\sigma^n\bigl( 1-[\mu_d]\bigr) \]
with $\mu_d$ carrying the obvious $\muu$-action. 

\begin{Theorem} \label{theorem2}
The motivic Donaldson--Thomas invariants for the one loop quiver with homogeneous potential $W=t^d$ are given by 
\[ \Omega_n= \begin{cases} 
              \LL^{-1/2}\bigl(1- [\mu_d]\bigr) &\mbox{ for } n=1, \\ 0 & \mbox{ else }.
             \end{cases}
\]
\end{Theorem}
\begin{proof}
 Using the previous lemma and our calculations,  we finally get
\begin{eqnarray*} \Sym\Biggl( \sum_{n\ge 1} \frac{\LL^n - 1}{\LL^{1/2}-\LL^{-1/2}} \Omega_n T^n\Biggr) & = & \Phi_W(\LL T)/\Phi_W(T) \\ & = &  \sum_{n\ge 0} \sigma^n \bigl(1 -[\mu_d]\bigr)T^n \\ & = &\Sym \Bigl( \big(1 -[\mu_d]\bigr)T \Bigr). 
\end{eqnarray*}
and the theorem follows by comparing coefficients.
\end{proof}
Let us now come back to the case of general potentials $0\neq W\in k[t]$ and let $W'=c\prod_{i_1}^r(t-a_i)^{d_i-1}$ be the prime decomposition of $W'$ into linear factors with $c\in k^\times$, $1< d_i\in \mathbb{N}$ and $a_i\in k$ for all $1\le i\le r$. Hence, the Grothendieck group of the abelian category of sheaves supported on the zero scheme of $W'$ is $\mathbb{Z}^r$ with effective cone $\mathbb{N}^r$ spanned by the classes $\vec e_i$ of skyscraper sheaves of length one supported at $a_i\in \AA^1_k$ $(i\in \{1,\ldots,r\})$ or equivalently by one-dimensional representations with eigenvector $a_i$. The monoid homomorphism $\dim: \Crit(\WW) \longrightarrow \mathbb{N}$ factorizes as $\dim: \Crit(\WW) \xrightarrow{\;\cl\;} \mathbb{N}^r \xrightarrow{\;+\;} \mathbb{N}$  with $\cl(V)$ being the class of the representation of $V$ in the Grothendieck group. This allows us to define refined Donaldson--Thomas invariants by means of 
\[ \int_{\cl} \phi_\WW = \Sym \Bigl( \frac{1}{\LL^{1/2} - \LL^{-1/2}} \sum_{\vec n=(n_1,\ldots,n_r) \in \mathbb{N}\setminus 0} \Omega_{\vec n}\, T_1^{n_1}\cdots T_r^{n_r} \Bigr) =: \Phi_W(T_1,\ldots,T_r) \]
in $\KK^{\muu}(\Sta/\mathbb{N}^r)$, where we denote the line element $[\Spec(k) \rightarrow \vec e_i]$ by $T_i$. For any $1\le i\le r$ let us write $\MM^{(i)}$  for the substack $\cl^{-1}(\mathbb{N}\vec e_i)$ parametrizing sheaves on $a_i$. Obviously 
\[\MM\cong \prod_{i=1}^r \MM^{(i)} \] 
by taking direct sums. However, to compute the motivic vanishing cycle $\phi_\WW$ we use different embeddings on each side. Indeed, the vanishing cycle on the left hand side restricted to $\cl^{-1}(n_1,\ldots,n_r)=:\MM_{\vec n}$ is computed by means of the embedding $\MM_{\vec n}\subset \Mat_k(N,N)/ \Gl_k(N)$ with $N=n_1+\ldots + n_r$, whereas on the right hand side we use the embedding 
\[ \prod_{i=1}^r \MM_{n_i \vec e_i} \subset \prod_{i=1}^r \Mat_k(n_i,n_i) /\Gl_k(n_i). \]
However, one can prove the product formula 
\begin{equation} \label{prodformula} \int_{\cl} \phi_\WW = \prod_{i=1}^r \int_{\cl} \phi_\WW|_{\MM^{(i)}}. \end{equation}
Indeed, as there are neither extensions nor morphisms between representations $V \in \MM^{(i)}$ and $V'\in \MM^{(j)}$ for $i\neq j$, we have 
\[ [\MM \rightarrow \MM ] = [\MM^{(1)} \rightarrow \MM] \star \ldots \star [\MM^{(r)}\rightarrow \MM] \] 
in the Hall algebra $\KK(\Sta/\MM)$. To prove formula (\ref{prodformula}) we apply the refined algebra homomorphism $\int_{\cl}^{\phi_{\WW}}$ from the Hall algebra to $\KK^{\muu}(\Sta/\mathbb{N}^r)$ (see \cite{Meinhardt3}). \\

For the computation of the right hand side in the product formula (\ref{prodformula}) we make use of the following lemma.
\begin{Lemma}
For any $i\in \{1,\ldots,r\}$ let $d_i$ be defined as above. Then
\[ \int_{\cl} \phi_W|_{\MM^{(i)}} = \Phi_{t^{d_i}}(T_i) = \Sym \biggl(\frac{1-[\mu_{d_i}]}{\LL-1}T_i\biggr).\] 
\end{Lemma}
\begin{proof}
Fix $i\in\{1,\ldots,r\}$ and $0<n\in\mathbb{N}$. By translation we can assume $a_i=0$ and $W=\sum_{p\ge 0}b_pt^{p+d_i}$ with $b_0\neq 0$ but $b_p=0$ for $p>>0$. Let $\tilde{W}(t)=t^{d_i}$. By solving the recursive equations
\[ \sum_{ m_0 + m_1 + \ldots = d_i \atop m_1 + 2m_2 + \ldots = p} { d_i \choose m_0, m_1, \ldots } a_0^{m_0}a_1^{m_1} \cdots = b_p \]
starting with a $d_i$-th root $a_0$ of $b_0$ one finds a power series $\theta(t)=t(a_0 + a_1t + \ldots )$ with $\theta(t)^{d_i}=\tilde{W}(\theta(t))=W(t)$. \\
The component $\MM_{n\vec e_i}$ is given by the quotient stack $C/\Gl(n)$ with $C=\{A\in \Mat_k(n,n) \mid A^{d_i-1}=0 \}$ as $dW_n(A):\Mat_k(n,n)\ni H \longmapsto \tr( H A^{d_i-1} G(A))\in \AA^1_k$ with $G\in k[t]$ having nonzero constant term. As the map $A\mapsto A^{m(d_i-1)}$ is in the $m$-th power of the defining ideal of $C$, we obtain a well defined isomorphism $\theta_n:\hat{C} \rightarrow \hat{C}$ on the formal neighborhood of $C$, i.e. the formal completion of $\Mat_k(n,n)$ along $C$. Moreover, the restriction of the function $\tilde{W}_n:A\mapsto \tr A^{d_i}$ to $\hat{C}$ composed with $\theta_n$ coincides with the restriction of $W_n$ to $\hat{C}$. As any arc of length $l$ in $\Mat_k(n,n)$ with constant term in $C$ is actually an arc of length $l$ in $\hat{C}$, $\theta_n$ induces isomorphisms $\Lo_l(\theta_n):\Lo_l(\Mat_k(n,n))|_C \xrightarrow{\:\sim\:} \Lo_l(\Mat_k(n,n))|_C$ with $\Lo_l(\tilde{W}_n)\circ \Lo_l(\theta_n)=\Lo_l(W_n)$ for any $l>0$. Hence $\int_C Z_{W_n}(T_i)=\int_C Z_{\tilde{W}_n}(T_i)$ with the notation from section 5 and $\int_C \phi_{W_n}=\int_C \phi_{\tilde{W}_n}$ follows proving the first equality. The second equality is a direct consequence of Theorem \ref{theorem2}.   
\end{proof}
Combining the lemma with the arguments before proves our main theorem.
\begin{Theorem}
For $W\in k[T]$ let $W'=c\prod_{i_1}^r(t-a_i)^{d_i-1}$ with $c\in k^\times$, $1< d_i\in \mathbb{N}$ and $a_i\in k$ for all $1\le i\le r$ as before. Define the Donaldson--Thomas invariants $\Omega_{\vec n}\in \KK^{\muu}(\Sta/k)$ for any $r$-tuple $(n_1,\ldots,n_r)\in \mathbb{N}^r$ as above. Then
\[ \Omega_{\vec n}= \begin{cases} 
              \LL^{-1/2}\bigl(1- [\mu_{d_i}]\bigr) &\mbox{ for } \vec n=\vec e_i\; (1\le i\le r),  \\ 0 & \mbox{ else }.
             \end{cases}
\]
In particular, $\Omega_{\vec n}$ is in the image of $\KK^{\muu}(\Var/k)[\LL^{-1/2}]$ in $\KK^{\muu}(\Sta/k)$.
\end{Theorem}

\nocite{JoyceHGF}
\nocite{JoyceCF}

\bibliographystyle{plain}
\bibliography{Literatur}

\vfill
\textsc{\small B. Davison: Mathematical Institute, University of Oxford, 24-29 St Giles', Oxford OX1 3LB, England}\\
\textit{\small E-mail address:} \texttt{\small Ben.Davison@maths.ox.ac.uk}\\
\\

\textsc{\small S. Meinhardt: Mathematisches Institut, Universit\"at Bonn, Endenicher Allee 60, 53115 Bonn, Germany}\\
\textit{\small E-mail address:} \texttt{\small sven@math.uni-bonn.de}\\

\end{document}